\documentclass [12pt]{amsart}
\usepackage[utf8]{inputenc}
\usepackage[all]{xy}
\usepackage[english]{babel}
\usepackage{amssymb, enumerate}
\usepackage{mathrsfs}
\usepackage[lite, initials]{amsrefs}
\usepackage{geometry}
\usepackage{xcolor}
\usepackage{dcpic,pictexwd}

\newtheorem{theorem}{Theorem}[section]
\newtheorem*{theorem*}{Main Theorem}
\newtheorem{lemma}[theorem]{Lemma}
\newtheorem{proposition}[theorem]{Proposition}
\newtheorem{corollary}[theorem]{Corollary}
\theoremstyle{definition}
\newtheorem{example}{Example}[section]
\newtheorem{remark}{Remark}[section]

\newtheorem*{warning*}{Warning}
\newtheorem{definition}[theorem]{Definition}
\newtheorem{assumption}[theorem]{Assumption}

\calclayout
\makeatletter %questo serve per fare una tableofcontent pi� gradevole
\def\@tocline#1#2#3#4#5#6#7{\relax
  \ifnum #1>\c@tocdepth % then omit
  \else
    \par \addpenalty\@secpenalty\addvspace{#2}%
    \begingroup \hyphenpenalty\@M
    \@ifempty{#4}{%
      \@tempdima\csname r@tocindent\number#1\endcsname\relax
    }{%
      \@tempdima#4\relax
    }%
    \parindent\z@ \leftskip#3\relax \advance\leftskip\@tempdima\relax
    \rightskip\@pnumwidth plus4em \parfillskip-\@pnumwidth
    #5\leavevmode\hskip-\@tempdima
      \ifcase #1
       \or\or \hskip 1em \or \hskip 2em \else \hskip 3em \fi%
      #6\nobreak\relax
    \dotfill\hbox to\@pnumwidth{\@tocpagenum{#7}}\par
    \nobreak
    \endgroup
  \fi}
\makeatother

\def\R{{\mathbb R}}
\def\N{{\mathbb N}}
\def\C{{\mathbb C}}

\def\Z{{\mathbb Z}}

\def\H{{\mathbb H}}
\def\cU{{\mathcal U}}

\def\sfera{{\mathbb S}}

\DeclareMathOperator{\pv}{\wedge \mspace{-9mu}_* \ }

\renewcommand{\Im}{\mathsf{Im}}

\def\Aut{\operatorname{Aut}}

\newcommand{\vecpart}[1]{\underline{#1}}
\newcommand{\vecnorm}[1]{\underline{#1}^2}

\title{The $*$-exponential as a covering map}
\author[A. Altavilla]{Amedeo Altavilla}\address{Altavilla Amedeo: Dipartimento di Matematica, Universit\`a degli Studi di Bari ``Aldo Moro'', via Edoardo Orabona, 4, 70125,
Bari, Italy.}\email{amedeo.altavilla@uniba.it}

\author[S. Mongodi]{Samuele Mongodi}\address{Mongodi Samuele: Dipartimento di Matematica e Applicazioni, Universit\`a degli Studi di Milano-Bicocca, via Roberto Cozzi, 55, 20125, Milano, Italy.}\email{samuele.mongodi@unimib.it}

\thanks{Partially supported by PRIN 2022MWPMAB - ``Interactions between Geometric Structures and Function Theories''
GNSAGA of INdAM and by the INdAM project ``Teoria delle funzioni ipercomplesse e applicazioni''.}


\begin{document}
\begin{abstract}
We employ tools from complex analysis to construct the $*$-logarithm of a quaternionic slice regular function. Our approach enables us to achieve three main objectives: we compute the monodromy associated with the $*$-exponential;
we establish sufficient conditions for the $*$-product of two $*$-exponentials to also be a $*$-exponential;
we calculate the slice derivative of the $*$-exponential of a regular function.
\end{abstract}

\maketitle
%\tableofcontents

\section{Introduction}

One of the fundamental ideas in the early stages of real analytic geometry and in the study of flat real analytic CR manifolds is the complexification: real analyticity ensures that most of the features we are interested in will be reflected in the complexification, at least on a formal and algebraic level.

In this respect, the theory of slice regular functions benefits no less, if not more, from this idea. As it is known, slice regular functions on a symmetric open domain $U\subseteq \H$ can be viewed as a special family of holomorphic functions (namely \textit{stem functions}) from an open domain in $\C$ to $\C\otimes\H$ (i.e. the \emph{complexification} of $\H$ as an associative algebra with unity).

In \cite{Mongodi:HolSR}, this idea was employed to account for the many similarities between the theory of slice regular functions and the theory of holomorphic functions of one complex variable; this argument was pushed further in \cite{Mongodi:AssAlg} to describe the link between the complex geometry of the set of square roots of $-1$ of an associative algebra and the space of slice regular functions.

On the other hand, the purely algebraic properties resulting from the structure of $\C\otimes\H$ were used to a great extent in~\cite{AltavillaPAMS,AltavillaAMPA,AltavillaLAA,altavillaLOG}, where many properties of the $*$-exponential
were studied, the structure of the $*$-product was better understood, and the problem of the existence of a $*$-logarithm
was initiated.

Building on the same intuition and exploiting more thoroughly the equality of the analytic expressions of algebraic operations of $\H$ and $\C\otimes\H$, in \cite{AM:powercover} we showed how the problem of finding $*$-roots of a slice regular function can be translated into a problem of lifting functions through a holomorphic covering map. The number and the structure of such $*$-roots were then linked to the group of deck transformations of the covering map.

Again, we would like to emphasize that this is possible because the analytic expressions of the multiplication in $\H$ and in $\C\otimes\H$ in terms of the coordinates with respect to some basis of $\H$ and the corresponding complexified basis of $\C\otimes\H$ are the same, reflecting the fact that a real analytic function on, say, the real line has a unique extension to the complexification of the real line given by a power series with the same coefficients.

We present here another instantiation of this consideration: we treat the case of $*$-logarithms by considering the map $\exp:\H\to\H$ and lifting it to the complexification, to a map with the same analytic expression. The study of the local inverses of $\exp$ again becomes a problem in complex analytic covering maps, from whose solution we also recover what we already proved in the case of $*$-roots. Using a geometric approach we will see that under natural topological hypotheses the exponential map in $\C\otimes\H$ is a covering map (see Theorem~\ref{vareps}) and we will be able to write down its monodromy. 
Then, thanks to the standard relation between holomorphic stem functions and slice regular functions, 
given a never-vanishing slice regular function $f:U\to\H$ such that its ``vector part'' is never-vanishing, we will be able
to construct a $2$-parameter family of $*$-logarithms ($1$-parameter family if $U\cap\R\neq \emptyset$), see
Corollaries~\ref{maincor} and~\ref{maincor2} for the results and Remark~\ref{monod} for the explicit description of
the monodromy.

Such a study extends what is already contained in~\cite{altavillaLOG,GPV1,GPV2} by showing the geometric nature of the many problems encountered in the search for a good notion of logarithm in the non-commutative setting.

As already mentioned, the proof of many results contained in the present paper follows topological strategies, which then produce natural hypotheses and conditions, simplifying many proofs contained in the aforementioned papers. On the other hand, since it is not the specific aim of this work, we will only give a glimpse of how the remaining residual cases should be treated, i.e. how some of the hypotheses could be relaxed.

In an effort to highlight the impact that this simple idea can produce, we analyze the problem of when a product of exponential is an exponential itself; the question for quaternions is easily settled by using a simplified version of the Baker-Campbell-Hausdorff formula (or, if one interprets quaternions as rotations, by a standard application of Rodrigues' formula).
Once the problem is \emph{analytically} solved for quaternions, we formally consider the same solution for the same problem in $\C\otimes\H$, where stem functions take their values. This gives us a solution to the same problem at the level of stem functions, hence for slice regular functions.
The same idea is used to compute the slice derivative of the $*$-exponential of a slice regular function. Even though this
computation is quite natural, it has not been implemented yet, possibly due to the lack of a strategy like the one we use here.
% for the lack of strategy as the one we use here. 
In these last two tasks, we will use a simple formula inspired by standard linear algebra, which allows, given a generic slice regular
function $f$, to write any other function $g$ as a sum of a component in the ``direction'' of $f$ and another ``orthogonal'' part.
This will allow us to write much simpler formulas and to identify possible future generalizations.

\section{Preliminaries}
\subsection{Algebraic structures of $\H$ and of $\H\otimes\C$}
In this paper we will deal with many different imaginary units, not only those contained in the space of quaternions, 
but also with others coming from different algebras. Starting from complex numbers, the symbol `$\imath$' will denote
the standard imaginary unit in $\C$ (and hence will be used when working in $\C^{N}$, $N\ge1$). 
The symbol `$i$' will denote the first imaginary unit in the definition of
the space $\H$ of quaternions:
$$
\H:=\{q=q_{0}+q_{1}i+q_{2}j+q_{3}k\,|\, q_{\ell}\in\R,\,\ell=0,1,2,3,\,i^{2}=j^{2}=k^{2}=-1,\,ij=k=-ji\}.
$$
We will make use of the standard conjugation in $\H$ denoted by the superscript $c$:
$$
q=q_{0}+q_{1}i+q_{2}j+q_{3}k\mapsto q^{c}=q_{0}-(q_{1}i+q_{2}j+q_{3}k).
$$
Using this conjugation, given any quaternion $q$, it is possible to define its scalar and vector parts as follows
$$
q_{0}=\frac{q+q^{c}}{2},\quad q_{v}=\frac{q-q^{c}}{2},
$$
so that $q=q_{0}+q_{v}$. Obviously, if $q$ is represented in the form $q=q_{0}+q_{1}i+q_{2}j+q_{3}k$, then $q_{0}$ is the scalar part of $q$ and
$q_{v}=q_{1}i+q_{2}j+q_{3}k$. 
Using this representation, we can express the product of two quaternions $q=q_{0}+q_{v}$ and $p=p_{0}+p_{v}$ in a more understandable way:
\begin{equation}\label{prod1}
qp=q_{0}p_{0}-\langle q_{v},p_{v}\rangle+q_{0}p_{v}+p_{0}q_{v}+q_{v}\wedge p_{v},
\end{equation}
where $\langle \cdot,\cdot\cdot\rangle$ and $\wedge$ denote the standard Euclidean and cross products.
In particular, the square norm of $q$ can be computed as $|q|^{2}=qq^{c}$ and we have that $q_{v}^{2}=-|q_{v}|^{2}$.
 
Whenever $q_{v}\neq0$, we are able to represent $q$ in another convenient form:
$$
q=\alpha+I\beta,
$$
where $\alpha=q_{0}$, $I=\frac{q_{v}}{|q_{v}|}$ and $\beta=|q_{v}|$.
In particular, if we denote the set of imaginary units as follows
$$
\sfera:=\{I\in\H\,|\,I^{2}=-1\}=\{\alpha_{1}i+\alpha_{2}j+\alpha_{3}k\,|\,\alpha_{1}^{2}+\alpha_{2}^{2}+\alpha_{3}^{2}=1\},
$$
and we denote by $\C_{I}=span(1,I)=\{\alpha+I\beta\,|\,\alpha,\beta\in\R\}$ the complex plane generated by $1$ and $I$, we have that
$$
\H=\bigcup_{I\in\sfera}\C_{I}.
$$
This last representation comes in handy when working with slice functions, and in order to do that we need to talk about
we need to discuss the complexification of $\H$ (in particular, we will follow the approach of~\cite{GP:AltAlg}).

The symbol `$\sqrt{-1}$' will denote the complex imaginary unit defining the complexification of $\H$, i.e.
$$
\C\otimes\H:=\{q+\sqrt{-1}p\,|\,q,p\in\H\}.
$$
The algebraic structure of $\C\otimes \H$ is defined in the usual way: if $q_{1}+\sqrt{-1}p_{1},q_{2}+\sqrt{-1}p_{2}\in\C\otimes\H$,
then:
$$
(q_{1}+\sqrt{-1}p_{1})(q_{2}+\sqrt{-1}p_{2})=q_{1}q_{2}-p_{1}p_{2}+\sqrt{-1}(q_{1}p_{2}+p_{1}q_{2}).
$$

By fixing a (orthogonal) basis of $\H$ containing $1$ and by writing any quaternion in its $4$ real coordinates
we get a biholomorphism between $\C\otimes \H$ and $\C^{4}$: if $q=q_{0}+q_{1}i+q_{2}j+q_{3}k$ and
$p=p_{0}+p_{1}i+p_{2}j+p_{3}k$, then we define $\phi:\C\otimes\H\to\C^{4}$ as
$$
\phi(q+\sqrt{-1}p)=(q_{0}+\imath p_{0},q_{1}+\imath p_{1}, q_{2}+\imath p_{2},q_{3}+\imath p_{3}).
$$
In particular, this biholomorphism induces an algebraic structure on $\C^{4}$ that is defined exactly as that of $\H$:
$$\C\otimes\H=\{z=z_{0}+z_{1}i+z_{2}j+z_{3}k\,|\,z_{\ell}\in\C\,,\ell=0,1,2,3,\,i^{2}=j^{2}=k^{2}=-1,\,ij=k=-ji\}.$$

In $\C\otimes\H$ it is possible to define two commuting conjugations:
\begin{align*}
q+\sqrt{-1}p\mapsto & (q+\sqrt{-1}p)^{c}=q^{c}+\sqrt{-1}p^{c},\\
q+\sqrt{-1}p\mapsto & \overline{q+\sqrt{-1}p}=q-\sqrt{-1}p.
\end{align*}
If we work in $\C^{4}$, these two conjugations translates as follows
\begin{align*}
z=z_{0}+z_{1}i+z_{2}j+z_{3}k\mapsto & z^{c}=z_{0}-(z_{1}i+z_{2}j+z_{3}k),\\
z=z_{0}+z_{1}i+z_{2}j+z_{3}k\mapsto & \overline{z}=\bar{z}_{0}+\bar{z}_{1}i+\bar{z}_{2}j+\bar{z}_{3}k.
\end{align*}
Exactly as before, we can define the ``scalar'' and ``vector'' part of $z\in\C\otimes\H$ as
$$
z_{0}=\frac{z+z^{c}}{2},\quad \vecpart{z}=\frac{z-z^{c}}{2},
$$

Within this language, the product of two elements $z,w\in\C\otimes\H$ can be written formally as in Formula~\ref{prod1}:
$$
zw=z_{0}w_{0}-\langle \vecpart{z},\vecpart{w}\rangle+z_{0}\vecpart{w}+w_{0}\vecpart{z}+\vecpart{z}\wedge\vecpart{w},
$$
where $\langle\cdot,\cdot\cdot\rangle$ and $\wedge$ are the formal generalization of the Euclidean and cross product.
In particular if $z=z_{0}+\vecpart{z}=z_{0}+z_{1}i+z_{2}j+z_{3}k$, setting 
$$\vecnorm{z}=z_{1}^{2}+z_{2}^{2}+z_{3}^{2},$$
we have that $zz^{c}=\langle z,z\rangle=z_{0}^{2}+\vecnorm{z}\in\C$ and it is a real number only if the four components
of $z$ are real numbers, i.e. only if $z\in\H$.
However, since the product in $\C$ is commutative, for any $z,w\in\C\otimes \H$ we have
\begin{equation}\label{formula1}
(zw)(zw)^{c}=zww^{c}z^{c}=(zz^{c})(ww^{c}).
\end{equation}
If $z$ is such that $zz^{c}\neq 0$, then $z\neq 0$, but unfortunately $\C\otimes\H$ contains zero divisors. So, in particular, 
there are $z\neq 0$ such that $zz^{c}=0$.
\subsection{Stem functions, slice functions and regularity}
We are now ready to introduce and discuss slice functions. As already said, we will rely on the approach of
stem functions developed in~\cite{GP:AltAlg} and in subsequent works by the same authors. We also refer to~\cite{AM:powercover} to deepen our specific point of view.
We start with the following definition.
\begin{definition}
Let $\mathcal{U}\subset\C$ be such that $\overline{\mathcal{U}}=\mathcal{U}$. A function $F:\mathcal{U}\to\C\otimes\H$
is said to be a \textit{stem function} if, for any $z\in\mathcal{U}$, we have $F(\bar z)=\overline{F(z)}$.
\end{definition}
If we write $F:\mathcal{U}\to\C\otimes\H$ as $F(z)=F_{ev}(z)+\sqrt{-1}F_{od}(z)$, then the condition $F(\bar z)=\overline{F(z)}$
is reflected in the following two equalities $F_{ev}(\bar z)=F_{ev}(z)$ and $F_{od}(\bar z)=-F_{od}(z)$. If, instead, we read $F$ as a function 
taking values in $\C^{4}$, $F(z)=(F_{0}(z),F_{1}(z),F_{2}(z),F_{3}(z))$, then the stem condition must be satisfied by 
all four components, i.e. $F_{\ell}(\bar z)=\overline{F_{\ell}(z)}$, for $\ell=0,1,2,3$.

\begin{definition}
Let $U\subset \H$ be such that if $q=\alpha+I\beta\in U$ then $\alpha+J\beta\in U$, for any $J\in\sfera$
and let $\mathcal{U}=\{\alpha+\imath\beta\,|\,\alpha+I\beta\in U\}$.
A function $f:U\to\H$ is said to be a \textit{slice function} if there exists a stem function $F=F_{ev}+\sqrt{-1}F_{od}:\mathcal{U}\to\C\otimes\H$
such that $f(\alpha+I\beta)=F_{ev}(\alpha+\imath\beta)+IF_{od}(\alpha+\sqrt{-1}\beta)$; in this case we will write $f=\mathcal{I}(F)$
and we will say that $f$ is induced by $F$.

If $U$ is a domain and $F$ is a holomorphic function, then $f$ is said to be a \textit{slice regular function}.
\end{definition}

The definition of stem functions guarantees the well-definition of slice functions: in fact, since $F(\bar z)=\overline{F(z)}$,
then the value of $f$ at $\alpha+(-I)(-\beta)$ is not different from that of $f$ at $\alpha+I\beta$. 
Examples of slice regular functions are polynomials and converging power series in the quaternionic variable $q$ with right quaternionic coefficients. 

The main property of slice
functions is the so-called \textit{Representation Formula} contained in the following statement (see~\cite[Theorem 1.16]{GSS:RegFunc}). It essentially says that
a slice function can be recovered from its values on two different semislices $\C_{I}^{+}$ and $\C_{K}^{+}$, where the apex `$+$'
indicates the upper half plane.
\begin{theorem}[Representation Formula]
Let $f:U\to\H$ be a slice function and let $J,K\in\sfera$ be such that $J\neq K$. Then, for every $\alpha+I\beta\in U$ the following
formula holds
$$
f(\alpha+I\beta)=(I-K)((J-K)^{-1}f(\alpha+J\beta))-(I-J)((J-K)^{-1}f(\alpha+K\beta)).
$$
\end{theorem}

It is well known that the pointwise product of two slice functions does not preserve regularity, however, the pointwise product
of two stem functions is a stem functions, therefore it is natural to introduce a new notion of product as follows.
\begin{definition}
Let $f=\mathcal{I}(F)$ and $g=\mathcal{I}(G)$ be two stem functions defined on the same domain $U$. The 
$*$\textit{-product} of $f$ and $g$ is defined as the slice function 
$$
f*g=\mathcal{I}(FG):U\to\H.$$
\end{definition}
Since the $*$-product is defined from the pointwise product in a non-commutative algebra (namely in $\C\otimes\H$),
it is non-commutative itself. However, if we consider a slice function $f=\mathcal{I}(F_{ev}+\sqrt{-1}F_{od}):U\to\H$, such that
$F_{ev}$ and $F_{od}$ take only real values, then $F$ is a legit complex function of one complex variable and, for any other
slice function $g$ defined on $U$, we have that
$$f*g=fg=g*f.$$
A function $f=\mathcal{I}(F)$ with the above property is said to be \textit{slice preserving}. In fact, as $F_{ev}$ and 
$F_{od}$ are real valued, then, for any $q=\alpha+I\beta\in U$, the element $f(q)$ belongs to the same slice $\C_{I}$
of $q$. Written as a complex curve in $\C^{4}$ the stem function $F$ of a slice preserving function
takes the following form
$$F(z)=(F_{ev}+\imath F_{od},0,0,0)=(F_{0},0,0,0).$$

At this stage we can apply all the formalism and properties described in the previous part of this Section and obtain that,
if $f=\mathcal{I}(F)$ and $g=\mathcal{I}(G)$, $F=F_{0}+F_{1}i+F_{2}j+F_{3}k$ and $G=G_{0}+G_{1}i+G_{2}j+G_{3}k$, $f_{\ell}=\mathcal{I}(F_{\ell})$ and $g_{\ell}=\mathcal{I}(G_{\ell})$ for $\ell=0,1,2,3$, $f_{v}=\mathcal{I}((F+F^{c})/2)$ and $g_{v}=\mathcal{I}((G+G^{c})/2)$, then
$$
f*g=f_{0}g_{0}-\langle f_{v},g_{v}\rangle_{*}+f_{0}g_{v}+g_{0}f_{v}+f_{v}\pv g_{v},
$$
where $\langle f_{v},g_{v}\rangle_{*}=f_{1}g_{1}+f_{2}g_{2}+f_{3}g_{3}$ and $f_{v}\pv g_{v}=(f_{2}g_{3}-f_{3}g_{2})i+(f_{3}g_{1}-f_{1}g_{3})j+(f_{1}g_{2}-f_{2}g_{1})k$ and, of course, all $f_{\ell}$ and $g_{\ell'}$ are slice preserving functions. 
These last two operators can be defined in an intrinsic way by means of the so-called regular conjugation:
given a slice function $f=\mathcal{I}(F):U\to\H$, we define its \textit{regular conjugate} as the function $f^{c}:U\to\H$
defined as $f^{c}=\mathcal{I}(F^{c})$. Then, if $g:U\to\H$ is another slice function, we have that
$$
\langle f,g\rangle_{*}=\frac{f*g^{c}+g*f^{c}}{2},\qquad f\pv g=f_{v}\pv g_{v}=\frac{f*g-g*f}{2}=\frac{[f,g]}{2}.
$$
This representation of the product highlights how many algebraic features of slice functions directly come from those
of quaternions (or of quaternionic curves).
For instance, two non slice preserving functions $f$ and $g$ commute if and only if $f_{v}\pv g_{v}\equiv 0$ if and only
if there exist two slice preserving functions $\alpha$ and $\beta$ not both identically zero, such that
$\alpha f_{v}+\beta g_{v}\equiv 0$ (see e.g.~\cite[Proposition 2.10]{AltavillaPAMS}). A particular instance of this
phenomenon is when $f$ and $g$ are both $\C_{I}$\textit{-preserving} for some $I\in\sfera$, i.e., for any $q=\alpha+I\beta$
in the domain of $f$ and $g$, we have that $f(q), g(q)\in\C_{I}$. Slice preserving functions are $\C_{I}$-preserving for
any $I\in\sfera$.
Keeping this parallelism between the algebraic features of $\H$ and those of $\C\otimes\H$, we notice that
the role of the Euclidean norm of $\R^{4}\simeq\H$ is taken here by the so-called \textit{symmetrization} of $f$:
given a slice function $f=f_{0}+f_{1}i+f_{2}j+f_{3}k$, its symmetrization is the function $f^{s}:=\mathcal{I}(FF^{c})=f_{0}^{2}+f_{1}^{2}+f_{2}^{2}+f_{3}^{2}$. The symmetrization of $f$ has an important role in the study of zeroes of $f$ (see~\cite[Chapter 3]{GSS:RegFunc}).

We now recall an important example of slice preserving regular function.
\begin{example}
Let $\mathcal{J}:\H\setminus\R\to\H$ be the function such that 
$$\mathcal{J}(q)=\frac{q_{v}}{|q_{v}|}.$$
This is clearly a slice preserving function and 	it is constant on each semislice $\C_{I}^{+}$.
In fact, if $q=\alpha+I\beta\in\H\setminus\R$, and $\beta>0$, then $\mathcal{J}(q)=I$.
This particular function plays an important role in the theory of slice regular functions with
domain that does not intersect the real axis. In fact, thanks to the fact that $\mathcal{J}*\mathcal{J}=\mathcal{J}^{2}\equiv -1$,
we are able to define slice regular idempotent functions (and hence zero divisors) as
$$
\ell_{+},\ell_{-}:\H\setminus\R\to\H,\qquad \ell_{\pm}=\frac{1\mp\mathcal{J}i}{2}.
$$
For a complete study of these functions see~\cite{AltavillaLAA}.
\end{example}

Starting from the previous example it is worth noticing that, given any slice regular function $f=f_{0}+f_{v}$,
such that $f_{v}\not\equiv 0$ and $\sqrt{f_{v}^{s}}$ is well defined (e.g. when $f_{v}^{s}$ is never vanishing see~\cite[Corollary 3.2]{AltavillaPAMS}),
then the function 
$$
\frac{f_{v}}{\sqrt{f_{v}^{s}}},
$$
is such that
$$
\frac{f_{v}}{\sqrt{f_{v}^{s}}}*\frac{f_{v}}{\sqrt{f_{v}^{s}}}\equiv -1\, .
$$
So, at least two (intrinsically) different functions, take the role of the imaginary unit in the setting of slice functions.

As said before, a slice regular function is a slice function such that its stem function is holomorphic.
In fact, if $f=\mathcal{I}(F):U\to\H$ is a slice function of class $\mathcal{C}^{1}$ defined on a domain $U$, then the function
$\partial F/\partial \bar z$ and $\partial F/\partial z$ are stem functions as well. In particular, the function $\partial_{c}f:=\mathcal{I}(\partial F/\partial z)$ is called \textit{slice derivative} of $f$. 
As it is clear from the definition, the slice derivative of a slice regular function controls the behavior ``along slices''. Thus,
to have complete information at first order of a slice regular function $f$ we need to consider another operator,
namely the \textit{spherical derivative}: given a slice function $f=\mathcal{I}(F):U\to\R$, we 
define $\partial_{s}f:U\setminus\R\to\R$ as the slice function $\partial_{s}f(\alpha+I\beta)=\mathcal{I}\left(\frac{F_{od}(\alpha+\imath \beta)}{\beta}\right)$. Even if it does not look like a derivative, the spherical derivative can also be obtained as the result
of a differential operator applied to $f$ (see~\cite{perotti}).

We close this preliminary section by recalling the definition of the $*$-exponential of a slice regular function.
\begin{definition}
Let $f:U\to\H$ be any slice function. We denote by $f^{*2}=f*f$ and, for any $N>2$ we define
$$
f^{*N}=f*f^{*(N-1)}.
$$
If $f$ is slice regular, then we define the function $\exp_{*}(f):U\to\H$ as
$$
\exp_{*}(f)=\sum_{n\in\N}\frac{f^{*n}}{n!}.
$$
\end{definition}
Many properties and representations of the $*$-exponential of a slice regular function are discussed in~\cite{AltavillaPAMS, altavillaLOG,GPV1,GPV2}.

\section{The quaternionic exponential as a covering map}
It is well known that the quaternionic exponential map $\exp:\H\to\H$ is a covering map. However, in order to be
self-contained we propose here a proof of this fact using its slice regular nature. In fact, the function $\exp$
is induced by the stem function $E:\C\to\C\otimes\H$ defined as $E(\alpha+\imath\beta)=e^{\alpha+\sqrt{-1}\beta}
=e^{\alpha}(\cos\beta+\sqrt{-1}\sin\beta)$
or, with our usual abuse of notation, viewed as a curve in $\C^{4}$, as $E(z)=(e^{z},0,0,0)$.
Notice that, from the definition, the function $\exp$ is slice preserving.

We now introduce the following family of sets where $\exp$ will result to be non-singular. For any $k\in\N$ set
$$\mathcal{U}_{k}:=\{q\in\H\,|\,k\pi<|q_{v}|<(k+1)\pi\}.$$
We are now ready to state and prove the first result.
\begin{theorem}
The real differential of $\exp$ is non-singular at $q$ if and only if $q\in\mathcal{U}_{k}$ for some $k\in\N$.
Moreover, for each $k\in\N$, the restriction $\exp_{|\mathcal{U}_{k}}$ is a diffeomorphism onto its image,
which is $\H\setminus\R$.
\end{theorem}

\begin{proof}
Following~\cite[Lemmas 3.1 and 3.3]{AM:powercover} or~\cite[Proposition 8.19]{GSS:RegFunc}, if $f:\Omega\to\H$ is slice regular
and $q=\alpha+I\beta\in\Omega$, then the real differential of $f$ is singular at $q$ if and only if 
$\partial_{c} f(q)=0$ or $\partial_{s} f(q)=0$ or $\partial_{c} f(q)(\partial_{s} f(q))^{c}\in (\C_{I})^{\perp}$.
We have that $\partial_{c} \exp=\exp$, so it is never vanishing. The spherical derivative of $\exp$ at $q=\alpha+I\beta$
is equal to $e^{\alpha}\frac{\sin\beta}{\beta}$ and therefore it vanishes if and only if $\beta=h\pi$, where $h\in\Z^{*}$.
As $((\partial_{c} \exp)(q))\cdot ((\partial_{s} \exp)(q))^{c}=(\exp(q))\cdot e^{\alpha}\frac{\sin\beta}{\beta}\in\C_{I}$, 
for all $q\in\H$, we have that the set of critical points, i.e. the set of points where the real differential has not maximum rank, is
given by 
$$C_{0}(\exp)=\{q\in\H\,|\, |q_{v}|=h\pi, h\in\Z^{*}\}=\{\pi(z,I)\,|\,\Im(z)=h\pi,\,h\in\Z^{*},\,I\in\sfera\}.$$

It is easy to see that $\exp(C_{0}(\exp))=\R$ and therefore, the set of singular points
$S(\exp)=\exp^{-1}(\exp(C_{0}(\exp)))$ is given by
\begin{equation}\label{sofexp}
S(\exp)=\{\pi(z,I)\,|\,\Im(z)=h\pi,\,h\in\Z,\,I\in\sfera\}.
\end{equation}
Collecting everything, we get the following equality
$$
\H\setminus S(\exp)=\bigcup_{h\in\N}\mathcal{U}_{h}.
$$

Now, as $\exp$ is slice-preserving and $\exp(\H\setminus S(\exp))\subseteq\H\setminus\R$,
given $q=\alpha+ I\beta, q'=\alpha'+I'\beta'\in\H\setminus S(\exp)$, we have that $\exp(q)=\exp(q')$ if $I=I'$.
In fact, if we denote by $\mathcal{T}:\H\setminus\R\to\H\setminus\R$ the map defined by
$$\mathcal{T}(\alpha+ I\beta)=\alpha+I(\beta+\pi)$$
for $\beta>0$ and $I\in\sfera$, by the standard properties of the complex exponential, 
we obtain that $q=\mathcal{T}^{(h)}(q')$ for some $h\in\N$, where the superscript denotes iterates.

More generally, $\mathcal{T}$ is a diffeomorphism from $\mathcal{U}_h$ and $\mathcal{U}_{h+1}$ for all $h$; therefore, $\exp$ is injective from $\mathcal{U}_h$ to $\H\setminus\R$ and, again by the properties of the complex exponential, also surjective. Being a bijective local diffeomorphism, $\exp$ is a diffeomorphism between $\mathcal{U}_h$ and $\H\setminus\R$.
\end{proof}

\begin{remark}
The open domains $\mathcal{U}_{h}$ are all disjoint, therefore the map 
$$\exp_{|\H\setminus S(\exp)}:\H\setminus S(\exp)\to\H\setminus\R$$ 
is a covering map in the trivial way (i.e. there is no ramification).
\end{remark}

\section{The exponential in $\C\otimes\H$}

In this section we are going to study the exponential function of the algebra $\C\otimes\H$. 
As $\C\otimes\H$ is biholomorphic to $\C^{4}$, this study is made in
order to apply the ideas of complex analysis to our quaternionic context. 
Given $z=z_{0}+z_{1}i+z_{2}j+z_{3}k\in\C^{4}\cong \C\otimes\H$ we recall that,
$\vecpart{z}=z_{1}i+z_{2}j+z_{3}k$ and $\vecnorm{z}=z_{1}^{2}+z_{2}^{2}+z_{3}^{2}$.
%\begin{warning}
%In what follows we will encounter at least three different complex structures that are usually denoted by `$i$'. We will try
%to keep different symbols as follows: the usual complex structure on $\C^{N}$, $N\geq 1$, will be denoted by $\imath$,
%the usual complex structure on $\C\otimes\H$ by $\sqrt{-1}$, while the first element of the standard basis of imaginary quaternions
%by $i$.
%\end{warning}
%
%\begin{warning}
%In the rest of the paper we will make use of the following abuse of notation: an element $z$ of $\C\otimes\H\cong\C^{4}$ will 
%be written both as $z=p+\sqrt{-1}q$, where $p,q\in\H$ or as $z=z_{0}+z_{1}e_{1}+z_{2}e_{2}+z_{3}e_{3}=z_{0}+\vecpart{z}$, where this second
%notation is preferred in order to highlight better the use of techniques from complex analysis.
%\end{warning}
Moreover, given a slice regular function $f=\mathcal{I}(F)$, we also recall the following set of relations already introduced in the preliminary section:
$$
f_{0}=\mathcal{I}(F_{0}),\quad f_{v}=\mathcal{I}(\vecpart{F}),\quad f^{c}=\mathcal{I}(F_{0}-\vecpart{F}),\quad f^{s}=\mathcal{I}(F_{0}^{2}+\vecnorm{F}),\quad f_{v}^{s}=\mathcal{I}(\vecnorm{F}).
$$
%\begin{itemize}
%\item $f_{0}=\mathcal{I}(F_{0})$;
%\item $f_{v}=\mathcal{I}(\vecpart{F})$;
%\item $f^{c}=\mathcal{I}(F_{0}-\vecpart{F})$;
%\item $f^{s}=\mathcal{I}(F_{0}^{2}+\vecnorm{F})$;
%\item $f_{v}^{s}=\mathcal{I}(\vecnorm{F})$.
%\end{itemize}

In~\cite{AM:powercover} we introduced the analog of the quaternionic $n$th $*$-power as
$\sigma_{n}:\C\otimes\H\to\C\otimes\H$, where
$$
\sigma_{n}(z)=(p_{0}^{n}(z_{0},\vecnorm{z}),z_{1}p_{1}^{n-1}(z_{0},\vecnorm{z}),z_{2}p_{1}^{n-1}(z_{0},\vecnorm{z}),z_{3}p_{1}^{n-1}(z_{0},\vecnorm{z})),
$$
where $p_{0}^{n}$ and $p_{1}^{n-1}$ are the usual Chebischev polynomials, such that
\begin{equation}\label{eqpot}
(x+iy)^{n}=p_{0}^{n}(x,y^{2})+ip_{1}^{n-1}(x,y^{2}).
\end{equation}

We are now able to introduce the exponential function of the algebra $\C\otimes\H$ as the function
$\varepsilon:\C\otimes\H\to\C\otimes\H$ defined by
$$
\varepsilon=\sum_{n=0}^{\infty}\frac{\sigma_{n}}{n!}\,.
$$

From the definition of $\sigma_{n}$, we have that
$$
\varepsilon(z)=\sum_{n=0}^{\infty}\frac{p_{0}^{n}(z_{0},\vecnorm{z})}{n!}+\vecpart{z}\sum_{n=0}^{\infty}\frac{p_{1}^{n-1}(z_{0},\vecnorm{z})}{n!}\,.
$$
Moreover, thanks to Formula~\ref{eqpot}, we have that
\begin{align*}
\sum_{n=0}^\infty\frac{p_0^n(x,y^2)}{n!}=\mathrm{Re}\sum_{n=0}^\infty\frac{(x+Iy)^n}{n!}&=e^x\cos(y)\, ,\\
y\sum_{n=0}^\infty\frac{p_1^{n-1}(x,y^2)}{n!}=\mathrm{Im}\sum_{n=0}^\infty\frac{(x+Iy)^n}{n!}&=e^x\sin(y)\,.
\end{align*}

The aim of the following pages is to prove that, under suitable hypotheses, $\varepsilon$ is a covering map.
In order to obtain such a result, we `lift' our construction to another space where we are able
to use standard techniques from complex analysis. This was the fruitful strategy already used in~\cite{AM:powercover}
in order to better understand $*$-roots of slice functions.

\subsection{Lift to $\C^{2}\times\mathcal{S}$}%\label{rho}
As in~\cite{AM:powercover} we consider the set $\mathcal{S}$ of imaginary units contained in $\C\otimes \H$
$$
\mathcal{S}:=\{z=\vecpart{z}\in\C\otimes\H\,|\,z^{2}=\vecnorm{z}=z_{1}^{2}+z_{2}^{2}+z_{3}^{2}=-1\},
$$
and the map $\rho:\C^{2}\times\mathcal{S}\to\C\otimes\H$ defined by
$$
\rho((u_{0},u_{1}),s)=u_{0}+u_{1}s.
$$
\begin{remark}
With the language and symbols of tensor product the set $\mathcal{S}$ contains those elements $z=p+\sqrt{-1}q$,
with $p,q\in\H$, such that $-1=(p+\sqrt{-1}q)^{2}=p^{2}-q^{2}+\sqrt{-1}(pq+qp)$. Therefore, $p$ and $q$ satisfy the following system 
(see also~\cite[Remark 4.2]{AM:powercover}),
$$ 
\begin{cases}
p^{2}-q^{2}=-1,\\
pq+qp=0.
\end{cases}
$$
\end{remark}

The map $\rho$ is a local diffeomorphism and,
if we set $\mathcal{W}':=\{(u_{0},u_{1})\in\C^{2}\,|\,u_{1}\neq 0\}$ and
$$
\Omega'=\rho(\mathcal{W}'\times\mathcal{S})=\{(z_{0},\vecpart{z})\in\C^{4}\,|\,\vecnorm{z}\neq0\},
$$
then, the restriction of $\rho$ to $\mathcal{W}'\times\mathcal{S}$ is, in fact, a double cover to its image
$\Omega'$. In fact, for any $(z_{0},\vecpart{z})\in\Omega'$, we have that
$$
\rho^{-1}(z_{0},\vecpart{z})=\left(\left(z_{0},\pm\sqrt{\vecnorm{z}}\right),\pm\frac{\vecpart{z}}{\sqrt{\vecnorm{z}}}\right).
$$

In~\cite{AM:powercover} we made a large use of this double cover because the total space is `large enough' 
to allow many classical properties to hold.
Now we consider the map $\mathfrak{e}:\C^2\times\mathcal{S}\to\C^2\times\mathcal{S}$ defined as 
$$\mathfrak{e}((u_0,u_1),s)=((e^{u_0}\cos(u_1),e^{u_0}\sin(u_1)),s).$$
This map is defined in order to have that $\varepsilon\circ\rho=\rho\circ\mathfrak{e}$,
$$
\begindc{\commdiag}[3]
\obj(0,180)[A]{$\C^2\times\mathcal{S}$}
\obj(280,180)[B]{$\C^2\times\mathcal{S}$}
\obj(0,0)[C]{$\C\otimes\H$}
\obj(280,0)[D]{$\C\otimes\H$}
\mor{A}{B}{$\mathfrak{e}$}[\atleft,\solidarrow]
\mor{A}{C}{$\rho$}[\atright,\solidarrow]
\mor{C}{D}{$\varepsilon$}[\atright,\solidarrow]
\mor{B}{D}{$\rho$}[\atleft,\solidarrow]
\enddc$$
i.e., $\mathfrak{e}$ can be viewed as the lift of $\varepsilon$ in $\C^{2}\times\mathcal{S}$.
We want to prove that $\mathfrak{e}$ is a covering map onto its image. 
This result will allow us to state that $\varepsilon$ is a covering map as well.
Before stating the result,
we introduce the following set
$$\mathcal{W}:=\{(u_{0},u_{1})\in\C^{2}\,|\,u_{0}^{2}+u_{1}^{2}= 0\}.$$
\begin{theorem}\label{frakcover}
The complex differential of $\mathfrak{e}$ is everywhere non-singular.
Moreover, the function $\mathfrak{e}$ is a covering map onto its image,
which is $\mathfrak{e}(\C^{2}\times\mathcal{S})=(\C^{2}\setminus\mathcal{W})\times\mathcal{S}$.
\end{theorem}

\begin{proof}
Setting $(u',s)=\mathfrak{e}(u,s)$, the differential of $\mathfrak{e}$ at $(u,s)$ is a map 
$$D\mathfrak{e}_{(u,s)}:T_{u}\C^{2}\times T_{s}\mathcal{S}\to T_{u'}\C^{2}\times T_{s}\mathcal{S}$$ 
that
can be represented by the following matrix
$$\begin{pmatrix}
e^{u_0}\cos(u_1)&-e^{u_0}\sin(u_1)&\mathbf{0}\\
e^{u_0}\sin(u_1)&e^{u_0}\cos(u_1)&\mathbf{0}\\
\mathbf{0}^{\top}&\mathbf{0}^{\top}&\mathbf{I_{\mathcal{S}}}
\end{pmatrix},$$
where $\mathbf{0}=(0,0)$ 
and, as $\dim_{\C}\mathcal{S}=2$, $\mathbf{I_{\mathcal{S}}}$ the $2\times 2$ identity matrix. 
We have $\det (D\mathfrak{e}_{(u,s)})=e^{2u_{0}}$, so $D\mathfrak{e}$ is always invertible, i.e. $\mathfrak{e}$
is a local diffeomorphism between $\C^{2}\times\mathcal{S}$ and itself. 

We now pass to look at the image of $\mathfrak{e}$. Given $(w_{0},w_{1})\in\C^{2}$ we consider the following system
\begin{equation}\label{system}
\begin{cases}
e^{u_{0}}\cos(u_{1})=w_{0}\\
e^{u_{0}}\sin(u_{1})=w_{1}.
\end{cases}
\end{equation}
We have that 
$$w_0+\imath w_1=e^{u_0}e^{\imath u_1},\qquad w_0-\imath w_1=e^{u_0}e^{-\imath u_1},$$
and hence
\begin{align*}
u_0&=\frac{\log(w_0+\imath w_1)+\log(w_0-\imath w_1)}{2}+ (h_{1}+h_{2})\imath\pi,\\
u_1&=\frac{\log(w_0+\imath w_1)-\log(w_0-\imath w_1)}{2\imath}+(h_{1}-h_{2})\pi,
\end{align*}
with $h_{1},h_{2}\in\Z$.
Therefore, the system in Formula~\eqref{system} has a solution if and only if $w_{0} \pm  \imath w_{1}\neq 0$, i.e. 
if and only if $w_{0}^{2}+w_{1}^{2}\neq 0$. Hence, $\mathfrak{e}(\C^{2}\times\mathcal{S})=(\C^{2}\setminus \mathcal{W})\times\mathcal{S}$.

We now pass to prove that $\mathfrak{e}$ is a covering map onto its image. 
Having proved that it is a local diffeomorphism, we are left to prove that the \textit{lifting property} is satisfied,
i.e., given a continuous curve $\gamma:[0,1]\to(\C^{2}\setminus\mathcal{W})\times \mathcal{S}$, we will show that it is possible
to construct a continuous $\tilde \gamma:[0,1]\to\C^{2}\times \mathcal{S}$ such that 
$\mathfrak{e}\circ\tilde\gamma=\gamma$, i.e., such that the following diagram commutes.
$$
\begindc{\commdiag}[3]
\obj(280,180)[B]{$\C^2\times\mathcal{S}$}
\obj(0,0)[C]{$[0,1]$}
\obj(280,0)[D]{$(\C^{2}\setminus \mathcal{W})\times\mathcal{S}$}
\mor{C}{B}{$\tilde\gamma$}[\atleft,\solidarrow]
\mor{C}{D}{$\gamma$}[\atright,\solidarrow]
\mor{B}{D}{$\mathfrak{e}$}[\atleft,\solidarrow]
\enddc$$

Given $((\tilde u_{0},\tilde u_{1}),\tilde s)\in\C^{2}\times\mathcal{S}$ let us consider a curve
$\gamma=((\gamma_{0},\gamma_{1}),\gamma_{s}):[0,1]\to(\C^{2}\setminus\mathcal{W})\times \mathcal{S}$,
such that $\gamma(0)=\mathfrak{e}((\tilde u_{0},\tilde u_{1}),\tilde s)$,
then $\gamma_{0}(t)^{2}+\gamma_{1}(t)^{2}\neq 0$ for all $t\in[0,1]$. Thus, if we define
$$\alpha(t)=\gamma_0(t)+\imath \gamma_1(t),\qquad \beta(t)=\gamma_0(t)-\imath \gamma_1(t),$$
we have that $\alpha$ and $\beta$ are continuous paths in $\C^{*}$. But then,
as $z\mapsto e^{z}$ is a covering map from $\C$ to $\C^{*}$, we can construct $\tilde\alpha$ and $\tilde\beta$
such that
$$e^{\tilde\alpha}=\alpha,\qquad e^{\tilde\beta}=\beta,$$
with $\tilde\alpha(0)=\tilde u_{0}+\imath \tilde u_{1}$ and $\tilde\beta(0)=\tilde u_{0}-\imath \tilde u_{1}$.

In conclusion, the path $\tilde \gamma:[0,1]\to \C^{2}\times\mathcal{S}$ defined by
$$
\tilde \gamma(t)=\left(\left(\frac{\tilde\alpha(t)+\tilde\beta(t)}{2},\frac{\tilde\alpha(t)-\tilde\beta(t)}{2}\right),\gamma_{s}(t)\right),
$$
is continuous, is such that $\tilde\gamma(0)=((\tilde u_{0},\tilde u_{1}),\tilde s)\in\mathfrak{e}^{-1}(\gamma(0))$
and $\mathfrak{e}\circ\tilde\gamma=\gamma$, i.e. the map $\mathfrak{e}:\C^{2}\times\mathcal{S}\to(\C^{2}\setminus\mathcal{W})\times\mathcal{S}$ is a covering map.
\end{proof}

\begin{remark}\label{monodromyfrak}
As a byproduct of the proof of Theorem~\ref{frakcover}, we get that the fundamental group of $\C^{2}\setminus\mathcal{W}$ is isomorphic to $\Z^{2}$; indeed, for $(h_{1},h_{2})\in\Z^{2}$, the monodromy is given by the following action
$$
(h_{1},h_{2})\cdot (u_{0},u_{1})=(u_{0}+ (h_{1}+h_{2})\imath\pi,u_{1}+(h_{1}-h_{2})\pi).
$$
Moreover, notice that as $h_{1}+h_{2}$ and $h_{1}-h_{2}$ have the same parity, this result is coherent to~\cite[Theorem 1.2]{altavillaLOG} (see the relation between the indices $m$ and $n$ at point \textit{(2)} in the second bullet of the referred result).
\end{remark}

\subsection{Back to $\C\otimes\H$}

We now move back to the study of $\varepsilon$. Let us recall from~\cite[Section 4]{AM:powercover} the definition of the following two sets:
\begin{align*}
V_{-1}&:=\{(z_{0},\vecpart{z})\in\C\otimes\H\,|\,z_{0}^{2}+\vecnorm{z}=0\}=\rho(\mathcal{W}\times\mathcal{S}),\\
V_{\infty}&:=\{(z_{0},\vecpart{z})\in\C\otimes\H\,|\,\vecnorm{z}=0\}=(\C\otimes\H)\setminus \Omega'.
\end{align*}
From these two sets, we define
$$
\Omega:=\varepsilon^{-1}(\C\otimes\H\setminus(V_{-1}\cup V_{\infty}))=\{(z_{0},\vecpart{z})\in\C\otimes\H\,|\,\vecnorm{z}\neq h^{2}\pi^{2}, \mbox{ for }h\in\Z\}.
$$
Notice that $\Omega$ is the exact transpose in the context of $\C\otimes\H$ of the set $\H\setminus S(\exp)$,
where $S(\exp)$ is defined in Formula~\eqref{sofexp}.
Eventually, recalling that $\varepsilon\circ\rho=\rho\circ\mathfrak{e}$ and collecting the fact that $\mathfrak{e}:\C^{2}\times\mathcal{S}\to(\C^{2}\setminus\mathcal{W})\times\mathcal{S}$ and $\rho:\mathcal{W}'\times\mathcal{S}\to\Omega'$ are covering maps, we have proved the following theorem.
\begin{theorem}\label{vareps}
The function $\varepsilon:\Omega\to\C\otimes\H\setminus(V_{-1}\cup V_{\infty})$  is a covering map and its monodromy group is isomorphic to $\Z^{2}$.
\end{theorem}
%
%We remark that in the previous Theorem, the set $V_{\infty}$ should be removed from the codomain because of the
%role of $\rho$ (compare with the beginning of Section~\ref{rho}).

The case when $h=0$ is somehow special as described in the following remark.

\begin{remark}
The further restriction of $\varepsilon$ to $V_{\infty}$ is a covering map onto its image $\varepsilon(V_{\infty})=V_{\infty}\setminus\{(0,\vecpart{z})\,|\,\vecnorm{z}=0\}$.
In fact, if $(z_{0},\vecpart{z})\in V_{\infty}$, i.e. when $h=0$, we have that 
$\varepsilon(z_{0},\vecpart{z})=e^{z_{0}}(1,\vecpart{z})$ (compare with~\cite[Corollary 4.6]{AltavillaPAMS}),
which, again, belongs to $V_{\infty}$; moreover it is easy to see that $(z_{0},\vecpart{z})\mapsto e^{z_{0}}(1,z)$ is a covering map. However, while in this case each element in $V_{\infty}$ has a one-parameter
family of preimages, thanks to Remark~\ref{monodromyfrak}, in the general case of Theorem~\ref{vareps} any element in $\C\otimes\H\setminus(V_{-1}\cup V_{\infty})$
has a two-parameters family of preimages in $\Omega$. 
So, in a sense, in the case described in this remark, we lose a bunch preimages. 
In particular, we have the following isomorphisms of the fundamental groups:
$$\pi_{1}(\C\otimes\H\setminus V_{\infty})\simeq\Z,\qquad \pi_{1}(\C\otimes\H\setminus (V_{-1}\cup V_{\infty}))\simeq\Z^{2}.$$
\end{remark}

Thanks to the previous theorem, we can construct global `logarithms' with respect to $\varepsilon$.

\begin{corollary}\label{main}
Let $\cU$ be a simply connected domain and let $F:\cU\to\C\otimes\H\setminus (V_{-1}\cup V_{\infty})$ be a continuous function. Then there exist a two-parameter family of continuous functions $F_{(h_{1},h_{2})}:\cU\to \C\otimes\H$, for $(h_{1},h_{2})\in\Z^{2}$, such that $\varepsilon\circ F_{h_{1},h_{2}}=F$.
\end{corollary}

Of course the previous corollary applies, in particular, to stem functions.

\section{Global $*$-logarithms}

In this short section we collect a series of consequences of the previous section, allowing us to define global $*$-logarithms 
of a slice regular function $f=\mathcal{I}(F)$ such that $f^{s}\neq 0\neq f_{v}^{s}$ or, equivalently, such that $F\in\C\otimes\H\setminus (V_{-1}\cup V_{\infty})$.
As in~\cite[Section 5]{AM:powercover} we declare the following assumption.
\begin{assumption}
From now on, the set of definition $\mathcal{U}=\overline{\mathcal{U}}$ of our stem functions will be open and simply connected
or the union of two simply connected domains (if $\mathcal{U}\cap\R=\emptyset$).
\end{assumption}
In~\cite{GPV1,GPV2} the quaternionic domains $U$ coming from the sets $\mathcal{U}$ just described in the previous
Assumption are called \textit{basic domains}.

Thanks to Corollary~\ref{main}, if $f$ is a slice function, such that its stem function $F$ does
not intersects $V_{-1}\cup V_{\infty}$, we virtually have a two-parameters countable family of $*$-logarithms.
We only need to check that the resulting logarithms at the level of $\C\otimes\H$ are stem functions as well.
Later we will see that if the domain of $f$ contains real points, then we lose a parameter, obtaining a 
closer analogy with the complex case.

As done in~\cite{AM:powercover} for the case of $n$-th $*$-powers, given a stem function 
$F:\cU\to\C\otimes\H\setminus(V_{-1}\cup V_{\infty})$, we
define the following set
$$
\mathcal{G}:=\{G:\cU\to\C\otimes\H\,|\,\varepsilon\circ G=F\}.
$$
Since $\varepsilon(\bar z)=\overline{\varepsilon(z)}$, if $G\in\mathcal{G}$ then
the function $\hat G:\cU\to\C\otimes\H$ defined by $\hat G(z)=\overline{G(\bar z)}$ belongs to
$\mathcal{G}$ as well, in fact
$$
(\varepsilon\circ \hat G)(z)=\varepsilon(\hat G(z))=\varepsilon(\overline{G(\bar z)})=\overline{\varepsilon(G(\bar z))}= \overline{F(\bar z)}=F(z).
$$
Note that $G$ is a stem function if and only if $\hat G=G$.
Therefore, we can prove the following result, the proof of which goes exactly as that of~\cite[Theorem 5.3]{AM:powercover}.
\begin{corollary}\label{maincor}
Let $U$ be a basic domain such that $U\cap\R=\emptyset$ and let $f:U\to\H$ be a slice regular function such that
$f^{s}(q)\neq 0\neq f_{v}^{s}(q)$, for all $q\in U$. Then, there exists a two-parameters family of slice functions
$f_{(h_{1},h_{2})}:U\to\H$, for $(h_{1},h_{2})\in\Z^{2}$, such that 
$$
\exp_{*}(f_{h_{1},h_{2}})=f.
$$
\end{corollary}

\begin{remark}\label{monod}
Following Remark~\ref{monodromyfrak}, if $g=g_{0}+g_{v}$ is a $*$-logarithm of $f$, then, for any couple of integers $h_{1}$ and $h_{2}$, the function 
$$g_{0}+(h_{1}+h_{2})\mathcal{J}\pi+(\sqrt{g_{v}^{s}}+(h_{1}-h_{2})\pi)\frac{g_{v}}{\sqrt{g_{v}^{s}}}=g+\left[(h_{1}+h_{2})\mathcal{J}+(h_{1}-h_{2})\frac{g_{v}}{\sqrt{g_{v}^{s}}}\right]\pi,$$ 
is a $*$-logarithm of $f$ as well (see~\cite[Theorem 1.2]{altavillaLOG}).
\end{remark}

We now pass to analyze the case in which the function $f$ is defined on a domain which intersects the real axis.
Under this hypothesis, it is clear that the function $\mathcal{J}$ cannot appear in the set of solutions.
In fact, as explained later, we will obtain that the two parameters $h_{1}$ and $h_{2}$ shall be related by the equality $h_{1}=-h_{2}$.

%@@ Nella versione precedente c'era un pezzo che ora e' commentato 
%
%questo enunciato con dimostrazione, ma non credo che adesso
%abbia tanto senso inserirlo, dato che sappiamo che le soluzioni esistono. Eliminerei da qui al prossimo @@
%\begin{theorem}
%If $U\cap\R\neq\emptyset$,  $G\in\mathcal{G}$ is such that $G(U\cap\R)\cap\R^4\neq \emptyset$ if and only if there exists $g:U\to\H$ such that $\mathcal{I}(g)=G$.
%\end{theorem}
%\begin{proof}
%Let $x^0\in U\cap\R$ such that $G(x^0)\in\R^4$.
%Now, $\tau(G)$ is another lift of $F$ via $\varepsilon$, but $\tau(G)(x^0)=\overline{G(\overline{x^0})}=G(x^0)$, therefore $\tau(G)(x^0)=G(x^0)$, which implies $\tau(G)=G$, by uniqueness of the lift through a covering map.
%Therefore, $G$ is a stem function.
%
%Conversely, if $G=\mathcal{I}(g)$, then $G=\tau(G)$ so $G(U\cap\R)\subseteq\R^4$.
%\end{proof}
%

The proof of the following corollary goes as that of~\cite[Theorem 5.4]{AM:powercover}.
\begin{corollary}\label{maincor2}
Let $U$ be a basic domain such that $U\cap\R\neq\emptyset$ and let $f:U\to\H$ be a slice regular function such that
$f^{s}(q)\neq 0\neq f_{v}^{s}(q)$, for all $q\in U$. Then, there exist a one-parameter family of slice functions
$f_{h}:U\to\H$, for $h\in\Z$, such that 
$$
\exp_{*}(f_h)=f.
$$
\end{corollary}

Exactly as in~\cite{AM:powercover}, thanks to our construction, the previous two corollaries can be stated without the hypothesis of regularity. 

\begin{remark}
As anticipated before, if the domain of $f$ contains real points a one-parameter of solution is missing.
In fact, if $x^0\in U\cap\R$, then $F(x^0)\in \R^4\subset\C^{4}\simeq\C\otimes\H$ and there exists 
$y^0\in\R^4$ such that $\varepsilon(y^0)=F(x^0)$. We have $y^0=\rho((u_0,u_1),s)$ with $(u_0,u_1)\in\R^2$.
Therefore, we obtain
\begin{align*}
\varepsilon^{-1}(F(x^0))&=\{\rho((h_{1},h_{2})\cdot((u_0,u_1),s))\,|\, (h_{1},h_{2})\in\Z^2\}\\
&=\{\rho((u_0+ (h_{1}+h_{2})\imath\pi,u_1+(h_{1}-h_{2})\pi),s)\,|\, (h_{1},h_{2})\in\Z^2\},
\end{align*}
whose only real points are those obtained when $h_{1}+h_{2}=0$ (see the first component), i.e. when $h_{2}=-h_{1}$. To each point $y$ in $\varepsilon^{-1}(F(x_0))$ we associate $G\in\mathcal{G}$ such that $G(x^0)=y$; the $G$'s described in the previous result are those corresponding to real points in $\varepsilon^{-1}(F(x^0))$.
\end{remark}

All the results contained in this paper so far are coherent with those contained in~\cite{altavillaLOG,GPV1}. As already pointed out in the introduction, the main difference here is the idea of using the complex geometry of $\C^{4}$ and of $\C^{2}\times\mathcal{S}$
to reveal the nature of $\varepsilon$ and of $\mathfrak{e}$ as covering maps. In a broad sense, the strategy of~\cite{altavillaLOG}
is that of ``solving'' the $*$-logarithm mostly in algebraic terms, while in~\cite{GPV1} the same problem is addressed by
considering a sort of ``$*$-logarithm variety'' and by analyzing the geometry of curves contained in it. 
As already noted in~\cite{AM:powercover}, in our opinion our approach seems to be more suitable to generalizations to other contexts,
while giving a global view of the geometric structure lying beneath the specific issue.

We conclude this section by highlighting how it is possible to recover results about the $*$-roots,
starting from the $*$-logarithm. The starting point is a quite standard argument from one complex variable but, as we will
see, the computation of the monodromy needs some deeper investigation.

\begin{remark}
As already said, in~\cite{AM:powercover} we widely studied the existence of a $n$-th $*$-rooth of a slice regular function.
Exactly as in the complex case, almost all the work done in~\cite{AM:powercover} can be recovered from the study
of the $*$-logarithm. In fact, if $U$ is a basic domain and $f:U\to\H$ is a slice function such that
$f^{s}(q)\neq 0\neq f_{v}^{s}(q)$, for all $q\in U$, then, for any $n\in\mathbb{R}$ we are able to define
$$
(f)^{*\frac{1}{n}}:=\exp_{*}\left(\frac{1}{n}\log_{*}(f)\right),
$$
where the apex $*$, means that we are considering the power with respect to the $*$-product.
We will show in the next section how to recover the monodromy of $n$-th $*$-root from that of the $*$-logarithm.
\end{remark}

%Moreover, when $n\in\mathbb{N}$ and $U\cap\R=\emptyset$, we have that $(f)^{*\frac{1}{n}}=(g)^{*\frac{1}{n}}$ if and only if
%$$
%\exp_{*}\left(\frac{1}{n}\log_{*}(f)\right)=\exp_{*}\left(\frac{1}{n}\log_{*}(g)\right).
%$$
%If we set $\tilde{f}=log_{*}(f)$ and $\tilde{g}=log_{*}(g)$, then  $\exp_{*}(\frac{1}{n}\tilde f)=\exp_{*}(\frac{1}{n}\tilde g)$ if and only if
%there are suitable integers $h_{1}$ and $h_{2}$ such that
%$$
%\frac{1}{n}\tilde g=\frac{1}{n}\tilde{f}_{0}+\mathcal{J}(h_{1}+h_{2})\pi+(\frac{1}{n}\sqrt{\tilde{f}_{v}^{s}}+(h_{1}-h_{2})\pi)\frac{\tilde{f}_{v}}{\sqrt{\tilde{f}_{v}^{s}}},
%$$
%that is
%$$
%\tilde{g}=\tilde{f}+(\mathcal{J}(h_{1}+h_{2})+(h_{1}-h_{2})\frac{\tilde{f}_{v}}{\sqrt{\tilde{f}_{v}^{s}}})n\pi.
%$$
%If we now apply the $*$-exponential on both sides, since $f$ and the function in the parenthesis commute, we get
%$$
%g=f\exp_{*}\left((\mathcal{J}(h_{1}+h_{2})+(h_{1}-h_{2})\frac{\tilde{f}_{v}}{\sqrt{\tilde{f}_{v}^{s}}})n\pi\right)=f\exp_{*}(\mathcal{J}(h_{1}+h_{2})n\pi)\exp_{*}\left((h_{1}-h_{2})\frac{\tilde{f}_{v}}{\sqrt{\tilde{f}_{v}^{s}}}n\pi\right)
%$$

\section{Automorphisms of $\mathfrak{e}$ and of $\varepsilon$}

In this section we will give a description of the deck transformations of $\mathfrak{e}$ and of $\varepsilon$, i.e. the 
automorphisms of $\C^{2}\times\mathcal{S}$ and of $\Omega$ fixing the fibers of $\mathfrak{e}$ and of $\varepsilon$, respectively. 
To be precise, given a covering map $\pi:X\to Y$, we are interested in the set $$
\Aut_{\pi}:=\{f:X\to X\,|\, \pi\circ f=\pi\}.
$$
In particular, we will study $\Aut_{\mathfrak{e}}$ and $\Aut_{\varepsilon}$. We recall from~\cite{AM:powercover}
that $\Aut_{\rho}=\{\mathrm{id},\Gamma\}$, where $\Gamma:\C^{2}\times\mathcal{S}\to\C^{2}\times\mathcal{S}$ is given
by 
$$\Gamma((u_{0},u_{1}),s)=((u_{0},-u_{1}),-s).$$
Thanks to the content of the previous section we are able to represent $\Aut_{\mathfrak{e}}$ in a convenient way.
In fact, let us define $T_{\ell}:\C^{2}\times \mathcal{S}\to\C^{2}\times\mathcal{S}$ as the function
$$
T_{\ell}((u_{0},u_{1}),s)=((u_{0}+\imath \pi,u_{1}+\ell\pi),s),\quad \ell=-1,1.
$$
Then, following the proof of Theorem~\ref{frakcover}, we have that  
%@@controllare la formula che segue. Prima c'era scritto $\Aut_{\mathfrak{e}}=\langle T_{1},T_{-1}\rangle_{\Z}$
$$
\Aut_{\mathfrak{e}}= \{T_{(a,b)}:=aT_{1}+bT_{-1}\,|\,a,b\in\Z,\, a\equiv_{2} b \}.
$$
In particular, given $(h_{1},h_{2})\in\Z^{2}$ from Remark~\ref{monodromyfrak}, we get $a=h_{1}+h_{2}$ and 
$b=h_{1}-h_{2}$, while, given $T_{(a,b)}$, then $h_{1}=\frac{a+b}{2}$ and $h_{2}=\frac{a-b}{2}$.

We now pass to study $\Aut_{\varepsilon}$. Recall that $\varepsilon\circ\rho=\rho\circ\mathfrak{e}$ and notice
that $\Gamma\circ T_{1}=T_{-1}\circ\Gamma$. Therefore, if $(a,b)\in\Z^{2}$, we get
$$\Gamma\circ (aT_{1}+bT_{-1})=a\Gamma\circ T_1 + b\Gamma\circ T_{-1}=a T_{-1}\circ\Gamma+ bT_1\circ\Gamma=(bT_1+aT_{-1})\circ\Gamma .$$
Hence, $T_{(a,b)}$ descend to a map $S\in\Aut_{\varepsilon}$, if and only if $(a,b)=(b,a)$, i.e. if and only if 
$T_{(a,b)}=T_{(a,a)}=a(T_{1}+T_{-1})$, $a\in\Z$. But then it follows that, if we define $S_{0}:\C\otimes\H\to\C\otimes\H$
as the map
$$
S_{0}(z_{0},\vecpart{z})=(z_{0}+2\imath \pi,\vecpart{z}),
$$
then
$$
\Aut_{\varepsilon}=\langle S_{0}\rangle_{\Z}\simeq\Z .
$$
We have just proven the following result which is analogous to~\cite[Corollary 6.4 and Proposition 6.5]{AM:powercover}
\begin{proposition}
The covering map $\mathfrak{e}$ is regular, while $\varepsilon$ is not.
\end{proposition}

%@@ A questo punto del file vecchio c'era una cosa che non capisco se ci serve. E' commentata
%{\bf To do :} Dimostrare che c'\'e un punto fisso, cio\'e che c'\'e una G tale che $\hat G=G
%@@

\subsection{Monodromy of $*$-roots}
We now want to recover the monodromy of $*$-roots by means of that of $*$-logarithm.
This computation was already performed in~\cite{AM:powercover} with different techniques. Here we will use
what we just learned from the study of $*$-logarithm.

Since, for any $z\in\C\otimes\H$, we have that $\varepsilon(2z)=\varepsilon(z)^{2}$, where sum and product are the 
algebra operations of $\C\otimes\H$, for any $n\in\N$, then we also have that
$$
\varepsilon(nz)=\sigma_{n}(\varepsilon(z)).
$$

Given $G\in\mathcal{G}$, let $H_G:\mathcal{U}\to \C^2\times\mathcal{S}$ be such that $\rho\circ H_G=G$. Then we have
$$\varepsilon\circ\rho\circ T_\ell\circ H_G=\rho\circ\mathfrak{e}\circ T_\ell\circ H_G=\rho\circ\mathfrak{e}\circ H_G=\varepsilon\circ\rho\circ H_G=\varepsilon\circ G=F\;.$$

%Given $T_{(a,b)}=aT_1+bT_{-1}$, we set $G_{(a,b)}=\rho\circ T_{(a,b)}\circ H_G$. Note that
%$$\widehat{G_{(a,b)}}(z)=\overline{\rho((u_0(\overline{z})+(a+b)\imath\pi,u_1(\overline{z})+(a-b)\pi),s(\overline{z}))}=$$
%$$=\rho(\overline{u_0(\overline{z})}+(-a-b)\imath\pi,\overline{u_1(\overline{z})} + (-b-(-a))\pi,\overline{s(\overline{z})})=\hat{G}_{(-b,-a)}(z)\;.$$
%If $G=\hat{G}$, then $\widehat{G_{a,b}}=G_{-b,-a}$, so, $\widehat{G_{a,-a}}=G_{a,-a}=(aS_0)\circ G$; which produces the stem functions given before.

Under the hypotheses of Corollary~\ref{main}, the set $\mathcal{G}$ is a two-parameters family of logarithms with
respect to $\varepsilon$. 
Therefore, we can represent each element of $\mathcal{G}$ as $G_{(a,b)}=\rho\circ T_{(a,b)}\circ H_{\tilde G}$, with $(a,b)\in\Z^{2}$,
where $G_{(0,0)}=:\tilde G$ is any particular solution of $\varepsilon\circ X=F$.

Now, $\varepsilon\circ (\frac{1}{n}{G}_{(a,b)})=\varepsilon\circ (\frac{1}{n}{G}_{(c,d)})$ with $(a,b),\ (c,d)\in\Z^2$, if and only if
$$\varepsilon\circ \rho\circ \frac{1}{n}T_{(a,b)}\circ H_{\tilde{G}}=\varepsilon\circ\rho\circ \frac{1}{n}T_{(c,d)}\circ H_{\tilde{G}}\;.$$
Set $H_{\tilde{G}}(z)=((u_0(z),u_1(z)),s(z))$, then
$$\frac{1}{n}T_{(a,b)}\circ H_{\tilde{G}}=\left(\left(\frac{1}{n}u_0(z)+\frac{a+b}{n}\imath\pi, \frac{1}{n}u_1(z)+\frac{a-b}{n}\pi\right), s(z)\right).$$
Assume, also, that $\{u_1\neq 0\}$, then the equality
$$\rho\circ \mathfrak{e}\circ \frac{1}{n}T_{(a,b)}\circ H_{\tilde{G}}=\rho\circ \mathfrak{e}\circ \frac{1}{n}T_{(c,d)}\circ H_{\tilde{G}}$$ 
holds true if and only if 
$$\mathfrak{e}\circ \frac{1}{n}T_{(a,b)}\circ H_{\tilde{G}}=\mathfrak{e}\circ \frac{1}{n}T_{(c,d)}\circ H_{\tilde{G}}$$
that is equivalent to say that there exists $(e,f)\in\Z^2$ such that
$$T_{(e,f)}\circ \frac{1}{n}T_{(a,b)}\circ H_{\tilde{G}}=\frac{1}{n}T_{(c,d)}\circ H_{\tilde{G}},$$
i.e.
\begin{multline*}
\left(\left(\frac{1}{n}u_0(z)+\left((e+f)+\frac{a+b}{n}\right)\imath\pi, \frac{1}{n}u_1(z)+\left((e-f)+\frac{a-b}{n}\right)\pi\right), s(z)\right)\\
=\left(\left(\frac{1}{n}u_0(z)+\frac{c+d}{n}\imath\pi, \frac{1}{n}u_1(z)+\frac{c-d}{n}\pi\right), s(z)\right),
\end{multline*}
i.e.
$$\frac{(c+d)-(a+b)}{n}=e+f\qquad \frac{(c-d)-(a-b)}{n}=e-f$$
$$\frac{c-a}{n}+\frac{d-b}{n}=e+f\qquad \frac{c-a}{n}-\frac{d-b}{n}=e-f$$
i.e.
$$\frac{c-a}{n}\in\Z\qquad \frac{d-b}{n}\in\Z\;.$$
i.e. $n|(c-a)$ and $n|(d-b)$, i.e. $(a,b)\equiv (c,d) \bmod n$.

Consider the subgroup of automorphisms
$$\mathsf{I}_{n}=\left\langle T_{(a,b)}\,|\, a\equiv b \equiv 0\bmod n\right\rangle$$
and let $W_n=\C^2\times\mathcal{S}/\mathsf{I}_n$; the projection $\mathfrak{e}_n:\C^2\times\mathcal{S}\to W_n$ is a covering map. As $\mathsf{I}_n$ is a subgroup of $\mathrm{Aut}_\mathfrak{e}$ and as $\mathfrak{e}$ is a Galois covering, we can factor $\mathfrak{e}$ via $\mathfrak{e}_n$: we consider the map $\mathfrak{s}_n:W_n\to (\C^2\setminus\mathcal{W})\times\mathcal{S}$ such that $\mathfrak{e}=\mathfrak{s}_n\circ\mathfrak{e}_n$; $\mathfrak{s}_n$ is in fact a covering map of degree $n^2$ (the index of $\mathsf{I}_n$ in $\mathrm{Aut}_\mathfrak{e}$).

As $\mathrm{Aut}_{\rho}$ does not intersect $\mathrm{Aut}_\mathfrak{e}$, we can induce a map $\tilde{\rho}:W_n' \to \Omega_n$ on a suitable open set $W'_n\subseteq W_n$ such that 
\begin{itemize}
\item $\tilde{\rho}$ is a double cover 
\item there is a (unique) covering map $\sigma_n:\Omega_n\to (\C\otimes\H)\setminus(V_{-1}\cup V_{\infty})$ with $\sigma_n\circ\tilde{\rho}=\rho\circ\mathfrak{s}_n$
\item there is a (unique) covering map $\epsilon_n:\Omega\to\Omega_n$ such that $\epsilon=\sigma_n\circ\epsilon_n$.
\end{itemize}

It is easy to notice that $\mathfrak{e_n}((u_0,u_1),s)=((u_0/n,u_1/n),s)$, $\epsilon_n(z)=\epsilon(z/n)$, $\tilde{\rho}=\rho$, $\sigma_n(z)=z^n$ (and the induced $\mathfrak{s}_n$) satisfy the previous requirements. Therefore, $W'_n$ and $\Omega_n$ can be realizes as subdomains of $\C^2\times\mathcal{S}$ and $\C\otimes\H$ respectively.

Carryig out the computations, one could find the definitions given in \cite{AM:powercover}.% [@@ da fare?] direi di no

Finally, we compute the monodromy of $\mathfrak{s}_n$. By simple arithmetic, the group $\Z_n\times\Z_n=\Z^2/\mathsf{I}$ is generated by the classes $[(1,1),(1,-1)]$ when $n$ is odd and by the classes $[(1,1),(1,-1),(1,0)]$ when $n$ is even.

Given $a,b\in\{0,\ldots, n-1\}$
$$\mathfrak{e}_n(T_{(a,b)}((u_0,u_1),s))=\mathfrak{e}\left(\left(\frac{1}{n}u_0+\frac{a+b}{n}\imath\pi, \frac{1}{n}u_1+\frac{a-b}{n}\pi\right),s\right)=$$
$$=\left(\left(e^{\frac{u_0}{n}}e^{\frac{a+b}{n}\imath\pi}\cos\left(\frac{u_1}{n} + \frac{a-b}{n}\pi\right),e^{\frac{u_0}{n}}e^{\frac{a+b}{n}\imath\pi}\sin\left(\frac{u_1}{n} + \frac{a-b}{n}\pi\right)\right),s\right)=\xi\cdot\left(A_\eta \frac{(u_{0},u_{1})}{n}, s\right)$$
where $\xi=e^{\frac{a+b}{n}\imath\pi}$ is a $n$-th root of unity and 
$$A_\eta=\begin{pmatrix}\cos(\frac{a-b}{n}\pi)&-\sin(\frac{a-b}{n}\pi)\\\sin(\frac{a-b}{n}\pi)&\cos(\frac{a-b}{n}\pi)\end{pmatrix}$$
is the $2\times2$ matrix representation of the complex number $\eta=e^{\frac{a-b}{n}\imath\pi}$.

So, for $n$ odd, the generators of the deck transformations of $\mathfrak{s}_n$ are $\xi$ (corresponding to $[(1,1)]$) and $A_\eta$ (corresponding to $[(1,-1)]$) with $\xi$, $\eta$ primitive $n$-th roots of $1$; for $n$ even we have these two and 
$\xi\cdot A_\eta$ (corresponding to $[(1,0)]$) with $\xi$, $\eta$ primitive $2n$-th roots of $1$.

\section{Product of two $*$-exponentials}
In this section we will give sufficient conditions for the product of two exponentials to be an exponential.
This topic clearly deals with the so-called Baker-Campbell-Hausdorff (BCH) formula for the $*$-exponential.

In its more general formulation the BCH formula states that, whenever it exists, the product $e^{X}e^{Y}$ equals $e^{Z}$,
where
\begin{equation}\label{BCH}
Z = X + Y + \frac{1}{2} [X,Y] + \frac{1}{12} [X,[X,Y]] - \frac{1}{12} [Y,[X,Y]] + \cdots
\end{equation}
Clearly, depending on the context, it is possible to give sufficient conditions for the sum on the right hand side of Formula~\eqref{BCH} to be convergent (see for instance~\cite[Proposition 2.2]{BiagiBCH} for Banach algebras or~\cite{bonfiglioli} for a general overview).
In the context of quaternions, the situation is much more clear: if $p=p_{0}+p_{v}$ and $q=q_{0}+q_{v}$, with $p_{v}\neq0\neq q_{v}$, then
\begin{align*}
\exp(p)\exp(q)=&\left[e^{p_{0}}\left(\cos|p_{v}|+\sin|p_{v}|\frac{p_{v}}{|p_{v}|}\right)\right]\left[e^{q_{0}}\left(\cos|q_{v}|+\sin|q_{v}|\frac{q_{v}}{|q_{v}|}\right)\right]\\
=&e^{p_{0}+q_{0}}\left[\cos|p_{v}|\cos|q_{v}|-\sin|p_{v}|\sin|q_{v}|\langle\frac{p_{v}}{|p_{v}|},\frac{q_{v}}{|q_{v}|}\rangle+\right. \\
&+\left.\cos|p_{v}|\sin|q_{v}|\frac{q_{v}}{|q_{v}|}+\cos|q_{v}|\sin|p_{v}|\frac{p_{v}}{|p_{v}|}+\sin|p_{v}|\sin|q_{v}|\frac{p_{v}}{|p_{v}|}\wedge\frac{q_{v}}{|q_{v}|}\right]\\
=&\exp(w_{0}+w_{v}),
\end{align*}
where $w_{0}=p_{0}+q_{0}$ and $w_{v}$ solves
$$
\begin{cases}
\cos|w_{v}|=\cos|p_{v}|\cos|q_{v}|-\sin|p_{v}|\sin|q_{v}|\langle\frac{p_{v}}{|p_{v}|},\frac{q_{v}}{|q_{v}|}\rangle,\\
\sin|w_{v}|\frac{w_{v}}{|w_{v}|}=\cos|p_{v}|\sin|q_{v}|\frac{q_{v}}{|q_{v}|}+\cos|q_{v}|\sin|p_{v}|\frac{p_{v}}{|p_{v}|}+\sin|p_{v}|\sin|q_{v}|\frac{p_{v}}{|p_{v}|}\wedge\frac{q_{v}}{|q_{v}|}.
\end{cases}
$$
Notice that, in this case, the existence of the solution is granted because $\exp(p)\exp(q)\neq 0$ and hence it is possible to 
define $w$.

Now, as $\H$ and $\C\otimes\H$ have the same algebraic structure, the same equalities hold true for the complexification.
However, recall that the euclidean norm must be translated into its purely algebraic form, i.e., if $u=u_{0}+\vecpart{u}, u'=u'_{0}+\vecpart{u'}\in\C\otimes\H$, then
\begin{align*}
\exp(u)\exp(u')=&e^{u_{0}+u'_{0}}\left[\cos\sqrt{\vecnorm{u}}\cos\sqrt{\vecnorm{u'}}-\sin\sqrt{\vecnorm{u}}\sin\sqrt{\vecnorm{u'}}\langle\frac{\vecpart{u}}{\sqrt{\vecnorm{u}}},\frac{\vecpart{u'}}{\sqrt{\vecnorm{u'}}}\rangle+\right. \\
&+\left.\cos\sqrt{\vecnorm{u}}\sin\sqrt{\vecnorm{u'}}\frac{\vecpart{u'}}{\sqrt{\vecnorm{u'}}}+\cos\sqrt{\vecnorm{u'}}\sin\sqrt{\vecnorm{u}}\frac{\vecpart{u}}{\sqrt{\vecnorm{u}}}+\right.\\
&+\left.\sin\sqrt{\vecnorm{u}}\sin\sqrt{\vecnorm{u'}}\frac{\vecpart{u}}{\sqrt{\vecnorm{u}}}\wedge\frac{\vecpart{u'}}{\sqrt{\vecnorm{u'}}}\right].
%=&\exp(p_{0}+\vecpart{p}).
\end{align*}
In this case, the solution $p=p_{0}+\vecpart{p}\in\C\otimes\H$ of the equation $\exp(u)\exp(u')=\exp(p)$ exists provided $\exp(u)\exp(u')\in \C\otimes\H\setminus (V_{-1}\cup V_{\infty})$. From the previous computations and also from~\cite[Theorem 4.14]{AltavillaPAMS}, we already know that if $u$ commutes with $u'$
or, if $\vecnorm{u}=\pi^{2}n^{2}$ and $\vecnorm{u'}=\pi^{2}m^{2}$ with $n$ and $m$ satisfying a certain parity condition, 
then $\exp(u)\exp(u')=\exp(u+u')$.
Therefore, we are interested in understanding when these conditions are satisfied excluding the cases
listed in the already mentioned result~\cite[Theorem 4.14]{AltavillaPAMS}. In order to proceed, we need a 
couple of preliminary lemmas.

\begin{lemma}\label{decomposition}
Let $\vecpart{z}\in\C\otimes\Im(\H)$ be such that $\vecnorm{z}\neq0$, then, for any $\vecpart{w}\in\C\otimes\Im(\H)$, there exist
$w_{1}\in\C$ and $w_{\perp}\in\C\otimes\Im(\H)$, such that 
$$
\vecpart{w}=w_{1}\vecpart{v}+w_{\perp},\qquad\mbox{and}\qquad\langle \vecpart{z},w_{\perp}\rangle=0.
$$
Moreover, it holds
$$
\vecnorm{(\vecpart{z}\wedge w_{\perp})}=\vecnorm{z}\vecnorm{w_{\perp}}.
$$
\end{lemma}
\begin{proof}
By standard linear algebra, it is sufficient to define 
$$w_{1}=\frac{\langle \vecpart{w},\vecpart{z}\rangle}{\langle \vecpart{z},\vecpart{z}\rangle}=\frac{\langle \vecpart{w},\vecpart{z}\rangle}{\vecnorm{z}}.$$
\end{proof}

Now, if $z,w\in\C\otimes\H$, then $(zw)_{0}^{2}+\vecnorm{(zw)}=(z_{0}^{2}+\vecnorm{z})(w_{0}^{2}+\vecnorm{w})$ (see Formula~\eqref{formula1}), therefore, if
$z_{0}^{2}+\vecnorm{z}\neq0\neq w_{0}^{2}+\vecnorm{w}$, then $(zw)_{0}^{2}+\vecnorm{(zw)}\neq0$. Hence, 
the following result complete the characterization we are looking for.

\begin{theorem}\label{prodvec}
Let $z,w\in\C\otimes\H\setminus (V_{-1}\cup V_{\infty})$, then, $\vecnorm{(zw)}=0$ if and only if
$$
w_{0}=-z_{0}w_{1}\pm \sqrt{-1}\sqrt{\frac{z_{0}^{2}+\vecnorm{z}}{\vecnorm{z}}}\sqrt{\vecnorm{w_{\perp}}},
$$
where $w_{1}$ and $w_{\perp}$ are the elements defined in Lemma~\ref{decomposition}.

%@@ non so se la formula e' l'espressione migliore: forse ha senso mettere $\sqrt{\vecnorm{z}}$ al numeratore? direi di no
\end{theorem}
\begin{proof}
First of all, thanks to Lemma~\ref{decomposition}, we can write $\vecpart{w}=w_{1}\vecpart{z}+w_{\perp}$, with $w_{1}\in\C$
and $\langle w_{\perp},\vecpart{z}\rangle=0$. Therefore,
\begin{align*}
zw&=(z_{0}+\vecpart{z})(w_{0}+w_{1}\vecpart{z}+w_{\perp})\\
&=(z_{0}w_{0}-\vecnorm{z}w_{1})+(z_{0}w_{1}\vecpart{z}+z_{0}w_{\perp}+w_{0}\vecpart{z}+\vecpart{z}\wedge w_{\perp}),
\end{align*}
and so
\begin{align*}
\vecnorm{(zw)}&=\vecnorm{(z_{0}w_{1}\vecpart{z}+z_{0}w_{\perp}+w_{0}\vecpart{z}+\vecpart{z}\wedge w_{\perp})}\\
&=\vecnorm{(z_{0}w_{1}\vecpart{z}+z_{0}w_{\perp}+w_{0}\vecpart{z})}+\vecnorm{(\vecpart{z}\wedge w_{\perp})}\\
&=(z_{0}w_{1}+w_{0})^{2}\vecnorm{z}+(z_{0}^{2}+\vecnorm{z})\vecnorm{w_{\perp}}.
\end{align*}
Therefore, $\vecnorm{(zw)}=0$ if and only if $(z_{0}w_{1}+w_{0})^{2}\vecnorm{z}+(z_{0}^{2}+\vecnorm{z})\vecnorm{w_{\perp}}=0$,
which is equivalent to
$$
(z_{0}w_{1}+w_{0})^{2}=-\frac{z_{0}^{2}+\vecnorm{z}}{\vecnorm{z}}\vecnorm{w_{\perp}},
$$
and hence, we get the thesis.
\end{proof}

The previous result can be applied to slice regular functions recalling their decomposition in ``scalar-vector'' parts.
We start by rewriting Lemma~\ref{decomposition} in terms of slice regular functions.
\begin{corollary}\label{decompositionf}
Let $f=f_{0}+f_{v}:U\to\H$ be a slice regular function such that, for any $q\in U$, $f_{v}^{s}(q)\not\equiv0$. Then, for any
slice regular function $g=g_{0}+g_{v}$, there exist two slice regular function $g_{1}, g_{\perp}$, $g_{1}$ being slice preserving,
such that 
$$
g=g_{1}f_{v}+g_{\perp},\qquad\mbox{and}\qquad\langle f_{v},g_{\perp}\rangle=0.
$$
Moreover, it holds
$$
(f_{v}\wedge g_{\perp})^{s}=f_{v}^{s}{g_{\perp}^{s}}.
$$
\end{corollary}
\begin{proof}
Exactly as in the previous result, for any $q\in U$ such that $f_{v}^{s}(q)\neq0$
 it is sufficient to define $g_{1}:=\frac{\langle g_{v},f_{v}\rangle}{f_{v}^{s}}$.
 Assume now that $q_{0}=\alpha_{0}+i\beta_{0}\not\in\R$ and $f_{v}^{s}(q_{0})=0$. Define $D_{q_{0}}(\epsilon)$ as the disk in $U\cap \C_{i}$
 centered at $q_{0}$ with radius $\epsilon$, such that $\overline{D_{q_{0}}(\epsilon)}\subset U\cap\C_{i}$ such that $f_{v}^{s}(q)\neq0$ for any $q\in \overline{D_{q_{0}}(\epsilon)}\setminus\{q_{0}\}$. Then we can define $g_{1}(q_{0})$ by means of the Cauchy formula
$$
 g_{1}(q_{0}):=\frac{1}{2\pi i}\int_{\partial D_{q_{0}}(\epsilon)}\frac{g_{1}(\alpha+i\beta)}{\alpha+i\beta-q_{0}}d(\alpha+i\beta).
$$
By repeating the same construction at $q_{0}^{c}$ and using the Representation Formula, we obtain the thesis.
The previous argument can be performed at $q_{0}\in\R$.
\end{proof}

\begin{corollary}\label{corovs}
Let $f,g$ be two slice regular functions defined on $U$ such that $f^{s}\neq0\neq g^{s}$ and $f_{v}^{s}\neq0\neq g_{v}^{s}$, then, $(f*g)_{v}^{s}(q)=0$ if and only if
\begin{equation}\label{formulaprodvec}
(f_{0}(q)g_{1}(q)+g_{0}(q))^{2}+\frac{f^{s}(q)}{f_{v}^{s}(q)}{g_{\perp}^{s}(q)}=0
\end{equation}
where $g_{1}$ and $g_{\perp}$ are the functions defined in Corollary~\ref{decompositionf}.
In particular, if $U\cap\R\neq\emptyset$, then there exists an open neighborhood $U$ of $U\cap\R$ such that
 ${(f*g)_{v}^{s}}_{|U}\neq0$.
% @@qui mettere $f_{v}^{s}(q)$ al numeratore, come corollary 6.2? direi di no
\end{corollary}
\begin{proof}
The first part of the statement is a direct consequence of Theorem~\ref{prodvec}. For the second part, assume that
$x\in U\cap\R$, then Formula~\eqref{formulaprodvec} evaluated at $x$, gives no solutions since the left hand side is
strictly positive. Therefore, there exists an open neighborhood $U$ of $U\cap\R$ where the function
$(f_{0}g_{1}+g_{0})^{2}+\frac{f^{s}}{f_{v}^{s}}{g_{\perp}^{s}}$ is never vanishing, and hence $(f*g)_{v}^{s}\neq0$ on $U$.
\end{proof}

Thanks to the previous two corollaries we can reverse engineer several examples of slice regular functions $f,g$ with 
$f^{s}\neq0\neq g^{s}$ and $f_{v}^{s}\neq0\neq g_{v}^{s}$ but $(f*g)_{v}^{s}(q)=0$.
\begin{example}
Assume for simplicity that $U\cap\R=\emptyset$. Then, given $f$ satisfying the hypotheses of previous corollary,
we define $g=g_{0}+g_{v}=g_{0}+g_{1}f_{v}+g_{\perp}$ as follows:
$$
g_{0}=-f_{0}g_{1}\pm\mathcal{J}\sqrt{\frac{f^{s}}{f_{v}^{s}}}\sqrt{g_{\perp}^{s}},
$$
with $g_{1}$ and $g_{\perp}$ be such that $g_{v}^{s}=g_{1}^{2}f_{v}^{s}+g_{\perp}^{s}\neq0$ and $g_{0}^{2}+g_{1}^{2}f_{v}^{s}+g_{\perp}^{s}\neq0$.
Clearly if $g_{\perp}^{s}\equiv 0$, then it is sufficient to take $g=g_{1}(-f_{0}+f_{v})+g_{\perp}$. For instance, if 
$f$ is $\C_{i}$-preserving, i.e. $f=f_{0}+f_{1}i$, with $f_{0}^{2}+f_{1}^{2}\neq0\neq f_{1}$, then, if $g=-f^{c}+\ell_{+,i}*j$,
we have that $f*g=-f^{s}+f*\ell_{+,i}*j$, $(f*g)_{v}=f*\ell_{+,i}*j$ and, therefore $(f*g)_{v}^{s}\equiv 0$.

Another readable case is when $f_{0}\equiv 0$. In this case, given $f=f_{v}$, it is sufficient to consider $g=\pm\mathcal{J}\sqrt{g_{\perp}^{s}}+g_{1}f_{v}+g_{\perp}$, with $g^{s}=g_{1}^{2}f_{v}^{s}\neq 0$, i.e. $g_{1}\neq0$, and $g_{1}^{2}f_{v}^{s}+g_{\perp}^{s}\neq 0$. 
\end{example}

Clearly, the previous example allows to construct functions $f$ and $g$, such that $(f*g)_{v}^{s}\equiv 0$, while the condition
in Corollary~\ref{corovs} si given point wise. 

At this point we are able to give sufficient conditions for $f$ and $g$ in order to have that $\exp_{*}(f)*\exp_{*}(g)$
is an exponential function. As we said, this happens if $(\exp_{*}(f)*\exp_{*}(g))_{v}^{s}(q)\neq 0$ for all $q$.
Clearly, we can separate the ``scalar'' part of $\exp_{*}(f)$ and of $\exp_{*}(g)$ and only consider 
$(\exp_{*}(f)_{v}*\exp_{*}(g)_{v})_{v}^{s}$.

We get the following result.
\begin{corollary}
Let $f,g:U\to\H$ be slice regular function such that $f_{v}$ do not commute with $g_{v}$ and for all $q\in U$
$f_{v}^{s}(q),g_{v}^{s}(q)\not\in\{\pi^{2}n^{2}\,|\,n\in\Z\}$. Write 
$g_{v}=g_{1}\frac{f_{v}}{\sqrt{f_{v}^{s}}}+g_{\perp}$. If for any $q\in U$
$$(g_{1}\cos\sqrt{f_{v}^{s}}\sin\sqrt{g_{v}^{s}}+\cos\sqrt{g_{v}^{s}})^{2}(\sin\sqrt{f_{v}^{s}})^{2}+(\sin\sqrt{g_{v}^{s}})^{2}g_{\perp}^{s}\neq0,$$
then there exists a slice regular function $h:U\to\H$ such that 
$$\exp_{*}(f)*\exp_{*}(g)=\exp_{*}(h).$$
\end{corollary}
Notice that in the last result the function $h=h_{0}+h_{v}$ is determined by 
$h_{0}=f_{0}+g_{0}$ and $h_{v}$ solves
$$
\begin{cases}
\cos\sqrt{h_{v}^{s}}=\cos\sqrt{f_{v}^{s}}\cos\sqrt{g_{v}^{s}}-\sin\sqrt{f_{v}^{s}}\sin\sqrt{g_{v}^{s}}\langle\frac{f_{v}}{\sqrt{f_{v}^{s}}},\frac{g_{v}}{\sqrt{g_{v}^{s}}}\rangle,\\
\sin\sqrt{h_{v}^{s}}\frac{h_{v}}{\sqrt{h_{v}^{s}}}=\cos\sqrt{f_{v}^{s}}\sin\sqrt{g_{v}^{s}}\frac{g_{v}}{\sqrt{g_{v}^{s}}}+\cos\sqrt{g_{v}^{s}}\sin\sqrt{f_{v}^{s}}\frac{f_{v}}{\sqrt{f_{v}^{s}}}+\sin\sqrt{f_{v}^{s}}\sin\sqrt{g_{v}^{s}}\frac{f_{v}}{\sqrt{f_{v}^{s}}}\wedge\frac{g_{v}}{\sqrt{g_{v}^{s}}}.
\end{cases}
$$

\section{Slice derivative of the $*$-exponential}
In this section we will provide a formula for the slice derivative of $\exp_{*}(f)$, $f$ being a slice regular function.
As in the previous section, let us begin with a short description of the general algebraic case.
If $X$ is a matrix, the differential of $e^{X}$, is given by the following formula
\begin{equation}\label{derivative}
e^{-X} de^X= dX-\frac{1}{2!}\left[X,dX\right]+\frac{1}{3!}[X,[X,dX]]-\frac{1}{4!}[X,[X,[X,dX]]]+\cdots
\end{equation}
Assume now that $q:[0,1]\to\H$ is a differentiable curve and denote by $\dot{q}=\frac{dq}{dt}$. Therefore, Formula~\ref{derivative}
can be written in the quaternionic setting as
\begin{align*}
e^{-q(t)} \frac{d e^{q(t)}}{dt}&= \dot{q}(t)-\frac{1}{2!}\left[q(t),\dot{q}(t)\right]+\frac{1}{3!}[q(t),[q(t),\dot{q}(t)]]-\frac{1}{4!}[q(t),[q(t),[q(t),\dot{q}(t)]]]+\cdots\\
&=\dot{q}(t)+\sum_{m=2}^{\infty}\frac{(-1)^{m-1}}{m!}[q^{(m-1)},\dot{q}](t),
\end{align*}
where $[q^{(n)},\dot{q}](t)$ stands for
$$
[\smash{\underbrace{q(t)[q(t)[\dots[q(t)}_{n\text{ times}}},\dot{q}(t)]]]].
$$
Therefore,
$$
e^{-q(t)} \frac{d e^{q(t)}}{dt}= \dot{q}(t)-\sum_{h=1}^{\infty}\frac{1}{(2h)!}[q^{(2h-1)},\dot{q}](t)+\sum_{h=1}^{\infty}\frac{1}{(2h+1)!}[q^{(2h)},\dot{q}](t).
$$
Now, as for any $p,q\in\H$, we have that $[p,q]=2p\wedge q=2p_{v}\wedge q_{v}$, then
we have
\begin{align*}
[q,\dot{q}]&=2q_{v}\wedge \dot q_{v}\\
[q^{(2)},\dot q]&=2^{2}[\langle q_{v},\dot q_{v}\rangle q_{v}-|q_{v}|^{2}\dot q_{v}]\\
[q^{(3)},\dot q]&=2^{3}(-1)|q_{v}|^{2} q_{v}\wedge \dot q_{v}\\
[q^{(4)},\dot q]&=2^{4}(-1)|q_{v}|^{2}[\langle q_{v},\dot q_{v}\rangle q_{v}-|q_{v}|^{2}\dot q_{v}]\\
[q^{(5)},\dot q]&=2^{5}(-1)^{2}(|q_{v}|^{2})^{2} q_{v}\wedge \dot q_{v}\\
[q^{(6)},\dot q]&=2^{6}(-1)^{2}(|q_{v}|^{2})^{2}[\langle q_{v},\dot q_{v}\rangle q_{v}-|q_{v}|^{2}\dot q_{v}]\\
&\dots
\end{align*}
where, in order to simplify the notation, we removed the dependence from the parameter $t$.
Hence,
\begin{align*}
e^{-q} \frac{d e^{q}}{dt}= &\dot{q}-\left[\sum_{h=1}^{\infty}\frac{(-1)^{h-1}2^{2h-1}}{(2h)!}|q_{v}|^{2(h-1)}\right]q_{v}\wedge \dot q_{v}\\
&+\left[\sum_{h=1}^{\infty}\frac{(-1)^{h-1}2^{2h}}{(2h+1)!}|q_{v}|^{2(h-1)}\right][\langle q_{v},\dot q_{v}\rangle q_{v}-|q_{v}|^{2}\dot q_{v}]\\
=&\dot{q}-\frac{\sin^{2}(|q_{v}|)}{|q_{v}|^{2}}q_{v}\wedge\dot{q}_{v}+\frac{|q_{v}|-\cos(|q_{v}|)\sin(|q_{v}|)}{|q_{v}|^{3}}\left[\langle q_{v},\dot q_{v}\rangle q_{v}-|q_{v}|^{2}\dot q_{v}\right]\\
=&\dot{q}-\frac{1-\cos(2|q_{v}|)}{2|q_{v}|}\frac{q_{v}}{|q_{v}|}\wedge\dot{q}_{v}+\left[1-\frac{\sin(2|q_{v}|)}{2|q_{v}|}\right]\left[\left\langle \frac{q_{v}}{|q_{v}|},\dot q_{v}\right\rangle \frac{q_{v}}{|q_{v}|}-\dot q_{v}\right]\\
=&\dot{q}+\left[1-\frac{\sin(2|q_{v}|)}{2|q_{v}|}\right]\left[\left\langle \frac{q_{v}}{|q_{v}|},\dot q_{v}\right\rangle \frac{q_{v}}{|q_{v}|}-\dot q_{v}\right]-\frac{1-\cos(2|q_{v}|)}{2|q_{v}|}\frac{q_{v}}{|q_{v}|}\wedge\dot{q}_{v}.
\end{align*}
Clearly, if $q_{v}=0$, we obtain the usual formula.
Moreover, if $q_{v}\neq0$ and $\dot{q}_{v}$ commutes with $q_{v}$, i.e. there exists a real valued function $\alpha$ such that $\dot{q_{v}}=\alpha q_{v}$,
then 
$$
e^{-q} \frac{d e^{q}}{dt}=\dot{q},
$$
as expected. 

If we write $\dot{q}=q_{1}\frac{q_{v}}{|q_{v}|}+q_{\perp}$, then we obtain
\begin{align}
e^{-q} \frac{d e^{q}}{dt}&=\dot{q}-\frac{1-\cos(2|q_{v}|)}{2|q_{v}|}\frac{q_{v}}{|q_{v}|}\wedge{q}_{\perp}-\left[1-\frac{\sin(2|q_{v}|)}{2|q_{v}|}\right] q_{\perp}\label{der1}\\
&=\dot{q}_{0}+q_{1}\frac{q_{v}}{|q_{v}|}
+\frac{\sin(2|q_{v}|)}{2|q_{v}|} q_{\perp}-\frac{1-\cos(2|q_{v}|)}{2|q_{v}|}\frac{q_{v}}{|q_{v}|}\wedge{q}_{\perp}.\label{der2}
\end{align}

Now, exactly as we have done before, this last relation extends to the complexification of $\H$, where
in place of a curve $q:[0,1]\to\H$ we consider a complex curve $F:D\subset\C\to\C\otimes\H$, the derivative with respect to $t$ is changed into the derivative with respect to $z\in D$ and
the usual exponential function is changed into $\varepsilon$. After these modifications we have the following formula.

\begin{proposition}
Let $f:U\to\H$ be a slice regular function. Then we have the following formula
\begin{align*}
\exp_{*}(f)^{-*}*\partial_{c}(\exp_{*}(f))=&\partial_{c}f+\left[\sum_{h=1}^{\infty}\frac{(-1)^{h-1}2^{2h}}{(2h+1)!}(f_{v}^{s})^{(h-1)}\right][\langle f_{v},(\partial_{c}f)_{v}\rangle_{*} f_{v}-f_{v}^{s}(\partial_{c}f)_{v}]+\\
&-\left[\sum_{h=1}^{\infty}\frac{(-1)^{h-1}2^{2h-1}}{(2h)!}(f_{v}^{s})^{(h-1)}\right]f_{v}\pv(\partial_{c}f)_{v}.
\end{align*}
\end{proposition}
From this proposition it is possible to derive some convenient corollaries.

\begin{corollary}
Let $f:U\to\H$ be a slice regular function and let $q_{0}\in U$ be any point. If $f_{v}^{s}(q_{0})=0$, then
$$
\Big(\exp_{*}(f)^{-*}*\partial_{c}(\exp_{*}(f))\Big)(q_{0})=(\partial_{c}f)(q_{0})-(f_{v}\pv(\partial_{c}f)_{v})(q_{0})+\frac{2}{3}\Big(\langle f_{v},(\partial_{c}f)_{v}\rangle_{*}\Big)(q_{0}) f_{v}(q_{0})
$$

\end{corollary}

%\begin{align*}
%\partial_{c}(\exp_{*}(f))=\exp_{*}(f)*&\left\{\partial_{c}f
%-\nu^{2}(f_{v}^{s})f_{v}\pv(\partial_{c}f)_{v}
%
%-\frac{1-\cos(2\sqrt{f_{v}^{s}})}{2\sqrt{f_{v}^{s}}}\frac{f_{v}}{\sqrt{f_{v}^{s}}}\pv(\partial_{c}f)_{v}+\right.\\
%&\left.+\left[1-\frac{\sin(2\sqrt{f_{v}^{s}})}{2\sqrt{f_{v}^{s}}}\right]\left[\left\langle \frac{f_{v}}{\sqrt{f_{v}^{s}}},(\partial_{c}f)_{v}\right\rangle \frac{f_{v}}{\sqrt{f_{v}^{s}}}-f_{v}^{s}(\partial_{c}f)_{v}\right]\right\}.
%\end{align*}

In the case when $f_{v}^{s}$ is never-vanishing, many equivalent formulas can be derived.
\begin{corollary}
Let $f:U\to\H$ be a slice regular function such that $f_{v}^{s}$ is never-vanishing, then we have
\begin{align*}
\partial_{c}(\exp_{*}(f))=\exp_{*}(f)*&\left\{\partial_{c}f+\left[1-\frac{\sin(2\sqrt{f_{v}^{s}})}{2\sqrt{f_{v}^{s}}}\right]\left[\left\langle\frac{f_{v}}{\sqrt{f_{v}^{s}}},(\partial_{c}f)_{v}\right\rangle\frac{f_{v}}{\sqrt{f_{v}^{s}}}-(\partial_{c}f)_{v} \right]+
\right.\\
&\left.
-\frac{1-\cos(2\sqrt{f_{v}^{s}})}{2\sqrt{f_{v}^{s}}}\frac{f_{v}}{\sqrt{f_{v}^{s}}}\pv(\partial_{c}f)_{v}
\right\}
\end{align*}

\end{corollary}
As said before, the formula contained in the last corollary is just one of the possible generalizations we have seen
in this section. With the same spirit, it is clearly possible to generalize Formula~\eqref{der1} or~\eqref{der2}.

\begin{remark}
Many of the previous formulas can also be related to the function $\nu:\H\to\H$ introduced in~\cite[Definition 2.16]{altavillaLOG} as
$$
\nu(q)=\sum_{m\in\N}\frac{(-1)^{m}q^{2m+1}}{(2m+1)!},
$$
and noticing that $\nu(q^{2})q=\sin(q)$.
\end{remark}

\begin{remark}
Exactly as in the case of a quaternionic curve, even in this case the formula for the slice derivative of
the $*$-exponential of a slice regular function simplifies to  the usual one when $(\partial_{c}f)_{v}$
and $f_{v}$ commute, i.e., getting rid of the trivial cases, when there exists a slice preserving function $\gamma$
such that 
$$
(\partial_{c}f)_{v}=\gamma f_{v}.
$$
Examples of functions satisfying this relations are slice constant functions (i.e. functions with everywhere vanishing 
slice derivative), $\C_{I}$-preserving function (for any $I\in\sfera$) or functions of the form
$f_{v}=\exp(\gamma(q)q)c$, where $c$ is any purely imaginary quaternion.
\end{remark}

\bibliographystyle{plain}
\bibliography{articolo_numero_due}{}

\end{document}